\documentclass[10pt]{amsart}
\RequirePackage[utf8]{inputenc}

\usepackage{amsmath,amsthm,amsfonts,amssymb,amscd,amsbsy,stmaryrd,dsfont,yfonts,hyperref}
\usepackage[all]{xy}
\usepackage{graphicx,pgf,caption,subcaption}

\newcommand{\dd}{\mathrm{d}}

\newcommand{\Ric}{\operatorname{Ric}}
\newcommand{\scal}{\operatorname{scal}}
\newcommand{\cut}{\operatorname{Cut}}
\newcommand{\g}{\mathrm g}
\newcommand{\h}{\mathrm h}
\newcommand{\cvc}{\operatorname{cvc}}
\newcommand{\spn}{\operatorname{span}}
\newcommand{\dist}{\operatorname{dist}}

\newcommand{\Id}{\operatorname{Id}}
\newcommand{\tr}{\operatorname{tr}}
\newcommand{\sff}{\mathrm{I\!I}}

\newcommand{\Z}{\mathds Z}
\newcommand{\R}{\mathds R}
\newcommand{\C}{\mathds C}
\newcommand{\SL}{\mathsf{SL}}

\allowdisplaybreaks

\hypersetup{
    pdftoolbar=true,
    pdfmenubar=true,
    pdffitwindow=false,
    pdfstartview={FitH},
    pdftitle={Three-manifolds with many flat planes},
    pdfauthor={Renato G. Bettiol and Benjamin Schmidt},
    pdfsubject={MSC Subject codes: 53B21, 53C20, 53C21, 53C24, 58A07, 58J60},
    pdfkeywords={},
    pdfnewwindow=true,
    colorlinks=true, % set this false for printable version
    linkcolor=blue,
    citecolor=blue,
    urlcolor=black,
}

\newtheorem{theorem}{Theorem}[]
\newtheorem{lemma}[theorem]{Lemma}
\newtheorem{proposition}[theorem]{Proposition}
\newtheorem{corollary}[theorem]{Corollary}

\newtheorem{mainthm}{\sc Theorem}

\theoremstyle{definition}
\newtheorem{definition}[theorem]{Definition}

\theoremstyle{remark}
\newtheorem{remark}[theorem]{Remark}

\newtheorem{convention}[theorem]{Convention}

\title{Three-manifolds with many flat planes}
\author[R. G. Bettiol]{Renato G. Bettiol}
\author[B. Schmidt]{Benjamin Schmidt}

\address{\begin{tabular}{lll}
University of Pennsylvania & & Michigan State University \\
Department of Mathematics & & Department of Mathematics \\
209 South 33rd St  & & 619 Red Cedar Road \\
Philadelphia, PA, 19104-6395, USA  & & East Lansing, MI, 48824, USA\\
\emph{E-mail address}: {\tt rbettiol@math.upenn.edu} & & \emph{E-mail address}: {\tt schmidt@math.msu.edu}
\end{tabular}
}

\allowdisplaybreaks
\numberwithin{equation}{section}
\numberwithin{theorem}{section}

\thanks{The first named author was partially supported by the NSF grant DMS-1209387. The second named author is partially supported by the NSF grant DMS-1207655.}
\subjclass[2010]{53B21, 53C20, 53C21, 53C24, 58A07, 58J60}

\date{April 25, 2016}

\begin{document}
\begin{abstract}
We discuss the rigidity (or lack thereof) imposed by different notions of having an abundance of zero curvature planes on a complete Riemannian $3$-manifold. We prove a rank rigidity theorem for complete $3$-manifolds, showing that having higher rank is equivalent to having reducible universal covering. We also study $3$-manifolds such that every tangent vector is contained in a flat plane, including examples with irreducible universal covering, and discuss the effect of finite volume and real-analiticity assumptions.
\end{abstract}

\maketitle
%\tableofcontents

\section{Introduction}
The Geometrization Conjecture and its resolution illustrate how closely the topology and the geometry of closed $3$-manifolds are related. However, many specific mechanisms through which curvature restricts the geometry of $3$-manifolds remain to be explored. In this paper, we are concerned with $3$-manifolds that have an abundance of flat tangent planes at every point. Namely, we study how the global arrangement of these flat planes constrains the geometry and topology of the underlying $3$-manifold.

A classical measure of how many flat planes a Riemannian manifold $M$ has is given by its \emph{rank}, defined as the least number of linearly independent parallel Jacobi fields along geodesics of $M$. 
Since the velocity field along a geodesic is a parallel Jacobi field, all manifolds have rank at least one. Consequently, $M$ is said to have \emph{higher rank} if it has rank at least $2$.
The geodesics of a manifold with rank $k\geq2$ infinitesimally appear to lie in a copy of $k$-dimensional Euclidean space.
The presence of parallel Jacobi fields encodes global information about the geometric arrangement of flat planes. Under so-called \emph{rank rigidity} conditions, these infinitesimal variations integrate to totally geodesic flat submanifolds, imposing a rigid structure on $M$.

Our first and main result is a characterization of higher rank $3$-manifolds:

\begin{mainthm}\label{thm:A}
A complete $3$-dimensional Riemannian manifold $M$ has higher rank if and only if its universal covering splits isometrically as $\widetilde M=N\times \R$.
\end{mainthm}

To our knowledge, this is the first rank rigidity theorem that only assumes completeness of the Riemannian metric. Besides not requiring any curvature bounds, the above result applies to both closed and open $3$-manifolds, including those of infinite volume. This flexibility is possible because we restrict to manifolds of low dimension. In fact, a simple argument shows that higher rank $3$-manifolds have \emph{pointwise signed sectional curvatures}; i.e., for each $p \in M$ the sectional curvatures of $2$-planes in $T_pM$ are either all nonnegative or all nonpositive (see Proposition~\ref{prop:signed}). Nevertheless, no global structural results for manifolds with $\sec\geq 0$ or $\sec\leq 0$ apply, since the sign of the curvature may change with the point. %the surface $N$ in Theorem~\ref{thm:A} having sectional curvatures of both signs.
The reader may consult \cite{kapo} for another instance in which flat surfaces in Riemannian $3$-manifolds are the source of rigidity in an otherwise curvature-free setting.

Historically, rank rigidity results first appeared in the context of manifolds with $\sec\leq0$. Ballmann~\cite{ba}, and independently Burns and Spatzier~\cite{busp}, proved that finite-volume manifolds with bounded nonpositive curvature and higher rank are either locally symmetric or have reducible universal covering. The lower bound on curvature was later removed by Eberlein and Heber~\cite{ebhe}, who also relaxed the finite-volume condition. Recently, Watkins~\cite{wa} proved a generalization of these results in the context of manifolds with no focal points. In contrast, fewer rank rigidity results are known for manifolds with lower sectional curvature bounds \cite{con,scshsp,scwo}, particularly $\sec\geq0$. A detailed investigation of rank rigidity in this curvature setting may have been stalled due to examples of Spatzier and Strake~\cite{spst} of $9$-dimensional manifolds with $\sec\geq0$ and higher rank that have irreducible universal covering and are not homotopy equivalent to a compact locally homogeneous space. It is worth mentioning that rank rigidity has also been extensively studied with different signs on sectional curvature bounds, with analogous definitions of \emph{spherical} and \emph{hyperbolic} rank, see \cite{co,esol,ha,shspwi}.

In order to understand the local structure of higher rank $3$-manifolds, we first analyze a \emph{pointwise} version of having higher rank. Following Schmidt and Wolfson~\cite{scwo}, we say that a Riemannian manifold $M$ has \emph{constant vector curvature zero}, abbreviated $\cvc(0)$, if every tangent vector to $M$ is contained in a flat plane. Higher rank manifolds clearly have $\cvc(0)$. Indeed, any $v\in T_pM$ is the initial velocity of the geodesic $\gamma_v(t)=\exp_p(tv)$, which, by the higher rank condition, admits a normal parallel Jacobi field $J(t)$ hence satisfying $\sec(\dot{\gamma}(t)\wedge J(t))=0$ for all $t\in\R$.
Notice that the $\cvc(0)$ hypothesis alone is a (strictly) weaker notion of having many flat planes, since there is no control on the global arrangement of these flat planes across the manifold, differently from the higher rank situation.

Nevertheless, $3$-manifolds with $\cvc(0)$ have a canonical decomposition as the union of the subset $\mathcal I$ of \emph{isotropic points}, at which all planes are flat, and its complement, the subset $\mathcal O$ of \emph{nonisotropic points}.
When $M$ has pointwise signed sectional curvatures, we have either $\sec\geq0$ or $\sec\leq0$ on each connected component $\mathcal C$ of $\mathcal{O}$. This imposes \emph{pointwise} rigidity on the geometric arrangement of flat planes, encoded as a singular tangent distribution on unit tangent spheres that we call \emph{flat planes distribution} (see Definition~\ref{def:flatplanesdistr}). The paucity of totally geodesic singular tangent distributions with codimension $\leq1$ on a round $2$-sphere (see Figure~\ref{fig:zeroplanesdistr}) is the source of such rigidity. In particular, it follows that $\mathcal C$ has a naturally defined line field $L$ such that a plane in $T_pM$, $p\in\mathcal C$, is flat if and only if it contains the line $L_p$, see Lemma~\ref{lemma:cvc}. From structural results in \cite{scwo}, $L$ is tangent to a foliation of $\mathcal C$ by complete geodesics, and $L$ is parallel if $\mathcal C$ has finite volume. Note that if $L$ is known to be parallel on $\mathcal C$, then the de Rham decomposition theorem provides a local product decomposition on $\mathcal C$.

Before outlining the strategy to prove our main result, let us mention a few examples of $3$-manifolds with $\cvc(0)$ that violate its conclusion. We start with examples of complete warped product metrics on $M\cong\R^3$ that have $\cvc(0)$ but are not isometric to a product metric, due to Sekigawa~\cite{se}. These examples are curvature homogeneous, with curvature model $\mathds{H}^2 \times\R$; in particular, they have $\sec\leq0$ and constant scalar curvature. In such examples, $M$ consists only of nonisotropic points and hence the line field $L$ is globally defined, but not parallel.

The classical rank $1$ graph manifolds of Gromov~\cite{gromov} provide examples of $3$-manifolds with $\cvc(0)$ and $\sec\leq0$. In these examples, the isotropic points separate the manifold in at least two components of nonisotropic points, where the line field $L$ is parallel. However, this line field is not the restriction of a globally defined parallel line field, and the local product structure given by the de Rham decomposition theorem cannot be globalized. This originates from the geometric arrangement of flat planes being locally compatible (within each nonisotropic component), but not globally compatible.
We observe that similar examples also exist among $\sec\geq0$. These can be constructed on both closed and open $3$-manifolds by gluing product manifolds with boundary, proving:

\begin{mainthm}\label{thm:B}
The sphere $S^3$, all lens spaces $L(p;q)\cong S^3/\Z_p$, and $\R^3$ admit complete Riemannian metrics with $\sec\geq0$, $\cvc(0)$, and irreducible universal covering.
\end{mainthm}

Informed by the above examples, we must use the additional information on the \emph{global} arrangement of flat planes to prove that a higher rank $3$-manifold $M$ has reducible universal covering. The two main steps are to show that the line field $L$ is parallel on each connected component of $\mathcal{O}$ and extendable to a globally defined parallel line field on $M$.
In order to implement this strategy, the key observation is that, on a higher rank manifold $M$, the domains $U_p=M \setminus \cut(p)$, $p\in\mathcal O$, admit a canonical \emph{open book decomposition} with totally geodesic flat pages and binding $\exp_p(L_p)$, see Figure~\ref{figure:flats}. Concretely, this decomposition is obtained by exponentiating the linear open book decomposition of $T_pM$ given by the union of $2$-planes that contain the line $L_p$, see Proposition~\ref{prop:flats}.

To achieve the first step, the open book decomposition is used to show that $L^\perp$ is a totally geodesic distribution (see Corollary~\ref{corollary:perptotgeo}). This, combined with an evolution equation for the shape operators of $L^{\perp}$ along $\exp_p(L_p)$, implies that $L$ is parallel in each nonisotropic component (see Proposition~\ref{prop:totgeoandparallel}).

To achieve the second step, note that the geodesic $\exp_p(L_p)$ generates a parallel line foliation of each flat page of the open book decomposition. The union of these parallel line foliations over pages extends $L$ across isotropic points to a parallel line field on $U_p$ (see Proposition~\ref{prop:keyprop2}).

As illustrated by the above mentioned examples, the higher rank hypothesis in Theorem~\ref{thm:A} cannot be relaxed to $\cvc(0)$, even under the additional assumption of pointwise signed sectional curvatures. We remark that the examples in Theorem~\ref{thm:B} use smooth metrics which are \emph{not real-analytic}, while the examples of Sekigawa~\cite{se} are real-analytic, but have \emph{infinite volume}. Our final result shows that if $M$ is a $3$-manifold with $\cvc(0)$ and pointwise signed sectional curvatures on which both of these behaviors are avoided, then the universal covering of $M$ is reducible.

\begin{mainthm}\label{thm:C}
Let $M$ be a complete $3$-dimensional real-analytic Riemannian manifold with finite volume and pointwise signed sectional curvatures. If $M$ has $\cvc(0)$, then its universal covering splits isometrically as $\widetilde M=N\times \R$.
\end{mainthm}

We conclude by remarking that, despite the above examples and rigidity results, there are no obvious topological obstructions for $3$-manifolds to admit a metric with $\cvc(0)$ and pointwise signed sectional curvatures. This is in sharp contrast with the stronger notion of $3$-manifolds with higher rank.

This paper is organized as follows. In Section~\ref{sec:prelim}, we establish basic notation and recall some well-known facts of Riemannian geometry. The structure of $3$-manifolds with $\cvc(0)$ and pointwise signed sectional curvatures is discussed in Section~\ref{sec:cvc0}. Section~\ref{sec:structure} contains our structural results on $3$-manifolds with higher rank. The proof of Theorem~\ref{thm:A} is given in Section~\ref{sec:proofA}. Examples of manifolds with $\cvc(0)$ without higher rank, including those mentioned in Theorem~\ref{thm:B}, are constructed in Section~\ref{sec:graphmflds}. Finally, the proof of Theorem~\ref{thm:C} is given in Section~\ref{sec:proofC}.

\subsection*{Acknowledgements}
We would like to thank the anonymous referee for the careful reading of our paper and thoughtful comments and suggestions.

\section{Preliminaries}\label{sec:prelim}

In this section, we recall some facts of Riemannian geometry and state a few preliminary results. Let $(M,\g)$ be a smooth connected $n$-dimensional Riemannian manifold. Its Levi-Civita connection $\nabla$ is determined by the Koszul formula:
\begin{equation}\label{eq:koszul}
\g(\nabla_{X} Y,Z)=\tfrac{1}{2}\big(\g([X,Y],Z)-\g([Y,Z],X)+\g([Z,X],Y)\big),
\end{equation} 
where $X$, $Y$ and $Z$ are orthonormal vector fields. Our sign convention for the curvature tensor is such that
\begin{equation}\label{eq:R}
R(X,Y)Z=\nabla_{X}\nabla_{Y} Z-\nabla_{Y}\nabla_{X} Z-\nabla_{[X,Y]} Z.
\end{equation}
For convenience, we also write $R(X,Y,Z,W):=\g(R(X,Y)Z,W)$. The sectional curvature of a $2$-plane spanned by orthonormal vectors $v,w\in T_pM$ is given by
\begin{equation*}
\sec(v \wedge w)=R(v,w,w,v).
\end{equation*}
We denote by $SM$ the unit sphere bundle of $M$, whose fiber $S_pM$ over $p \in M$ is the unit sphere of $T_pM$.

\subsection{Jacobi operator}
Given $v \in S_pM$, define the \emph{Jacobi operator} $\hat{\mathcal J}_v\colon T_pM \to T_pM$, $\hat{\mathcal J}_v(w)=R(w,v)v$. From the symmetries of the curvature tensor, it follows that $\hat{\mathcal J}_v$ is self-adjoint and $\hat{\mathcal J}_v(v)=0$. In particular, $\hat{\mathcal J}_v$ restricts to a self-adjoint operator $\mathcal J_v\colon v^{\perp} \to v^{\perp}$, $\mathcal J_v(w)=R(w,v)v$, on the subspace $v^{\perp}$ of vectors orthogonal to $v$.
%Let $\lambda_1<\lambda_2 <\cdots<\lambda_k$ denote the eigenvalues of $\mathcal J_v$, with corresponding eigenspaces $E_i$, so that $\sec(v\wedge e_i)=\lambda_i$ for any unit vector $e_i \in E_i$. Notice that $R(e_i,v,v,e_j)=0$ for any $e_i \in E_i$, $e_j \in E_j$ with $i\neq j$.
Elementary arguments yield the following well-known fact:

\begin{lemma}\label{lemma:crit}
Assume that $p \in M$ is such that either $\sec_p\geq0$ or $\sec_p\leq0$, and let $v,w \in S_pM$ be orthonormal vectors. Then $w\in\ker\mathcal J_v$ if and only if $\sec(v\wedge w)=0$.
\end{lemma}

%\begin{proof}
%Assume $\sec_p \geq 0$ (the case $\sec_p\leq0$ is completely analogous). It is clear that if $\mathcal J_v(w)=0$, then $\sec(v\wedge w)=0$. Conversely, let $\{e_1,\dots,e_{n-1}\}$ be an orthonormal basis of $v^\perp$ such that $\mathcal J_v(e_i)=\lambda_i e_i$. Since $\sec\geq0$, it follows that $\lambda_i\geq0$. Writing $w=\sum_{i=1}^{n-1} w_i e_i$, we have that $0=\sec(v\wedge w)=R(w,v,v,w)=\g(\mathcal J_v(w),w)=\sum_{i=1}^{n-1} \lambda_i w_i^2$.  Thus, $w_i=0$ whenever $\lambda_i>0$, which implies $\mathcal J_v(w)=0$.
%\end{proof}

\subsection{Jacobi fields}
Recall that a vector field $J(t)$ along a geodesic $\gamma(t)$ is a \emph{normal Jacobi field} along $\gamma$ if it is perpendicular to $\dot{\gamma}(t)$ and satisfies the \emph{Jacobi equation}
$J''+\mathcal J_{\dot{\gamma}}(J)=0$.
Normal Jacobi fields along $\gamma$ are infinitesimal variations of $\gamma$ by other geodesics with the same speed. As a solution of an ODE, any normal Jacobi field is determined by its initial conditions $J(0),J'(0) \in \dot{\gamma}(0)^{\perp}$.

\begin{convention}
Given $v\in S_{p}M$, denote by $\gamma_v(t):=\exp_p(tv)$ the geodesic with $\gamma_v(0)=p$ and $\dot\gamma_v(0)=v$. All geodesics are assumed parametrized by arc length.
\end{convention}

Given $v\in S_pM$ and $w \in v^{\perp}$, the map $\alpha(s,t)=\exp_p(t(v+sw))$ defines a variation of $\gamma_v(t)=\alpha(0,t)$ by geodesics, with variational field $J(t)=\frac{\partial}{\partial s}\alpha(s,t)\big|_{s=0}$ given by
\begin{equation}\label{jabfield}
J(t)=\dd(\exp_p)_{tv}(tw),
\end{equation}
which is a normal Jacobi field along $\gamma_v$ with initial conditions $J(0)=0$ and $J'(0)=w$. Conversely, all normal Jacobi fields $J(t)$ along $\gamma_v$ with $J(0)=0$ are of this form.

%Recall that the point $q=\exp_p(t_* v)$ is \emph{conjugate} to $p$ along $\gamma_v$ if there exists a normal Jacobi field $J(t)$ along $\gamma_v$ such that $J(0)=0$ and $J(t_*)=0$. Equivalently, $q$ is conjugate to $p$ along $\gamma_v$ if $t_* v\in T_pM$ is a critical point of $\exp_p\colon T_pM\to M$. Moreover, $q$ is the \emph{first} conjugate point to $p$ along $\gamma_v$ if there are no normal Jacobi fields $J(t)$ along $\gamma_v$ with $J(0)=0$ that also vanish for some $0<t<t_*$; or, equivalently, if $tv\in T_pM$ is not a critical point of $\exp_p$ for all $0<t<t_*$.

\subsection{Cut locus}
Henceforth, assume $(M,\g)$ is complete, and let $\dist\colon M\times M\to \R$ be its distance function. Geodesics in $M$ locally minimize distance, hence $\dist(p,\gamma_v(t))=t$ for each $v \in S_pM$ and $t>0$ sufficiently small. The \emph{cut time} $\mu(v)\in (0,+\infty]$ of $v \in S_pM$ is the maximal time for which $\gamma_v$ is minimizing:
\begin{equation}\label{eq:cut}
\mu(v):=\sup\{t>0 : \dist(p,\gamma_v(t))=t\}.
\end{equation}
If $\mu(v)< \infty$, the point $\gamma_v(\mu(v))$ is called the \emph{cut point} to $p$ along $\gamma_v$. Let $\cut(p)\subset M$ be the \emph{cut locus} of $p\in M$, defined as the set of cut points to $p$ along some $\gamma_v$ with $v\in S_pM$. It is well-known that if $q \in \cut(p)$, then either $q$ is the first conjugate point to $p$ along some geodesic, or there are at least two minimizing geodesics joining $p$ to $q$. In particular, $q \in \cut(p)$ if and only if $p \in \cut(q)$. It is also well-known that $\cut(p)$ is a closed subset of measure zero in $M$. We denote its complement by
\begin{equation}\label{eq:injball}
U_p:=M \setminus \cut(p).
\end{equation}
Using the above properties of $\cut(p)$, it is easy to prove the following:

\begin{lemma}\label{lemma:cover}
If $\mathcal O \subset M$ is an open subset, then $M=\bigcup_{p \in \mathcal O} U_p.$
\end{lemma}

%\begin{proof}
%Let $x\in M$ and $q\in\mathcal O$. If $x\notin U_q$, then $x \in \cut(q)$ and hence there exists a minimizing geodesic segment $\gamma\colon [0,\dist(q,x)] \to M$, with $\gamma(0)=q$ and $\gamma(\dist(q,x))=x$, that cannot be extended to a longer minimizing segment. Since $\mathcal O$ is open, there exists $0<\varepsilon<\dist(q,x)$ such that $p:=\gamma(\varepsilon) \in \mathcal O$. As the restriction of $\gamma$ to $[\varepsilon,\dist(p,q)]$ can be extended beyond $\varepsilon$ to a minimizing geodesic segment, $p\notin\cut(x)$. Thus, $x \notin \cut(p)$ and hence $x \in U_p$.
%\end{proof}
 
%\subsection{de Rham splitting}
%The \emph{de Rham decomposition} (or \emph{splitting}) is a classical result in Riemannian geometry, which states that a complete simply-connected manifold whose tangent bundle decomposes as direct sum of $k$ subbundles invariant under parallel translation must be isometric to a product manifold with $k$ factors. As a direct consequence, we have the following:
%
%\begin{proposition}\label{prop:derham}
%Let $(M,\g)$ be a complete $n$-dimensional Riemannian manifold. If $(M,\g)$ admits a parallel line field $X$, then its universal covering splits isometrically as $\widetilde M=N\times\R$.
%\end{proposition}

\subsection{Tangent distributions}
A tangent distribution $\mathcal D$ on a manifold $M$ is an assignment of a linear subspace $\mathcal D_p\subset T_pM$ for each $p\in M$. If $\dim\mathcal D_p$ is constant, then $\mathcal D$ is said to be \emph{regular}; otherwise, $\mathcal D$ is said to be \emph{singular}. The tangent distribution $\mathcal D$ is \emph{smooth} on an open subset $U$ of $M$ if there exist smooth vector fields $X_i$ on $U$ such that $\mathcal D_p=\spn\{X_i(p)\}$ for all $p\in U$.
If a vector field $X$ satisfies $X(p)\in \mathcal D_p$ for all $p\in M$, then we write $X\in\mathcal D$.
A tangent distribution $\mathcal{D}$ on $M$ is \emph{integrable} if there exists a submanifold $\mathcal F_p$ through each $p\in M$ such that $\mathcal D_p=T_p\mathcal F_p$. If $\mathcal D$ is a smooth regular tangent distribution, then the Frobenius Theorem states that $\mathcal D$ is integrable if and only if $\mathcal D$ is \emph{involutive}, i.e., $[X,Y]\in\mathcal D$ whenever $X,Y\in\mathcal D$. In this case, the partition of $M$ by the submanifolds $\mathcal F_p$ is called a \emph{regular foliation} of $M$. However, this characterization is no longer true for singular tangent distributions and singular foliations.\footnote{For \emph{smooth} singular tangent distributions, integrability to a singular foliation is equivalent to a condition strictly stronger than involutivity, requiring also completeness of the vector fields that span the distribution, see Sussmann~\cite{sus}.}

A (possibly singular) tangent distribution $\mathcal D$ on a Riemannian manifold is \emph{totally geodesic} if geodesics that are somewhere tangent to $\mathcal D$ remain everywhere tangent to $\mathcal D$. It is not hard to see that $\mathcal D$ is totally geodesic if and only if
\begin{equation}\label{eq:totgeodD1}
\g(\nabla_Z Z,\eta)=-\g(\nabla_Z \eta,Z)=0\quad \mbox{ for all } Z\in\mathcal D,\, \eta\in\mathcal D^\perp.
\end{equation}
Writing $Z=X+Y$, it follows that the above is, in turn, equivalent to
\begin{equation}\label{eq:totgeodD}
\g(\nabla_{X} Y+\nabla_{Y} X,\eta)=0\quad \mbox{ for all } X,Y\in\mathcal D,\, \eta\in\mathcal D^\perp.
\end{equation}
Notice that when $\mathcal D^\perp$ is integrable, $\mathcal D$ is totally geodesic if and only if $\mathcal D^\perp$ is tangent to a so-called \emph{(singular) Riemannian foliation}.

In the particular case in which $\mathcal D$ is regular and integrable, condition \eqref{eq:totgeodD} is clearly equivalent to the vanishing of the \emph{second fundamental form}
\begin{equation}\label{eq:secondff}
\sff\colon\mathcal D\times\mathcal D\to\mathcal D^\perp, \quad \sff(X,Y)=(\nabla_X Y)^\perp,
\end{equation}
where we denote by $Z=Z^\top+Z^\perp$ the components of a vector $Z$ tangent to $\mathcal D$ and $\mathcal D^\perp$ respectively. Indeed, \eqref{eq:totgeodD} corresponds to $\g(\sff(X,Y),\eta)$ being skew-symmetric; and, if $\mathcal D$ is integrable, this bilinear form is also symmetric by the Frobenius Theorem and hence vanishes. Nevertheless,  \eqref{eq:secondff} does not necessarily vanish when $\mathcal D$ is a general totally geodesic tangent distribution.

For each normal direction $\eta\in\mathcal D^\perp$ consider the \emph{shape operator}
\begin{equation}\label{eq:shapeop}
S^\eta\colon\mathcal D\to\mathcal D, \quad S^\eta(X)=-(\nabla_X \eta)^\top.
\end{equation}
Provided $\mathcal D$ is regular and integrable, $S^\eta$ is a self-adjoint operator which represents the bilinear form $\g(\sff(X,Y),\eta)$. In particular, such a distribution $\mathcal D$ is totally geodesic if and only if $S^\eta$ vanishes for all $\eta\in\mathcal D^\perp$.
Again, \eqref{eq:shapeop} does not necessarily vanish when $\mathcal D$ is a general totally geodesic tangent distribution.

\section{\texorpdfstring{Structure of $3$-manifolds with $\cvc(0)$ and signed curvatures}{Structure of 3-manifolds with cvc(0) and signed curvatures}}\label{sec:cvc0}

Let $M$ be a $3$-manifold with $\cvc(0)$ and pointwise signed sectional curvatures. In other words, we assume that for each $p \in M$, either $\sec_p \geq 0$ or $\sec_p \leq 0$; and, for all $v\in S_pM$, there exists $w\in S_pM$ such that $\sec_p(v\wedge w)=0$. In this section, we study the local structure of such a manifold $M$, as a preliminary step towards understanding the global structure of higher rank manifolds in the next section.

\subsection{Flat planes distribution}
Since $M$ has $\cvc(0)$, the subspace $\ker\mathcal J_v$ of $v^\perp$ is nontrivial for all $v\in S_pM$, see Lemma~\ref{lemma:crit}. This suggests analyzing the geometry of the collection of subspaces $\ker\mathcal J_v$ parametrized by $v\in S_pM$, cf.\ Chi~\cite{chi}.

\begin{convention}
For each $v\in S_pM$, parallel translation in $T_pM$ defines a canonical isomorphism between the subspace $T_v (S_pM)$ of $T_v (T_pM)$ and the subspace $v^{\perp}$ of $T_pM$. This identification will be made without mention when unambiguous.
\end{convention}

\begin{definition}\label{def:flatplanesdistr}
The \emph{flat planes distribution} $\mathcal Z$ is the tangent distribution on the unit tangent sphere $S_pM$ given by the association 
\begin{equation}\label{eq:flatplanesdistr}
v\mapsto\mathcal Z_v:=\ker\mathcal J_v.
\end{equation}
In other words, $\mathcal Z_v$ consists of all the $w$'s orthogonal to $v$ for which $v\wedge w$ is a zero curvature $2$-plane. We denote by $\textgoth{Z}_p$ its \emph{regular set}, which is the open subset of $S_pM$ formed by those $v$'s for which $\dim \mathcal Z_v=1$.
\end{definition}

\begin{remark}\label{rem:chi}
For each $p\in M$, the restriction of the flat planes distribution to $\textgoth Z_p$ is a smooth distribution. This follows from the fact that the Jacobi operator $\mathcal J_v$ varies smoothly with $v$ and has constant rank in $\textgoth Z_p$, hence its kernel also varies smoothly with $v$, cf. \cite[Lemma 1]{chi}. %A similar argument implies that the flat planes distribution is smooth on the subset of $SM$ consisting of rank $2$ vectors.
\end{remark}

\begin{convention}
Henceforth, unit tangent spheres $S_pM$ are assumed to be endowed with the standard (round) Riemannian metric induced by the metric $\g_p$ on $T_pM$. To avoid confusion, geodesics in $S_pM$ are typically denoted by $c$, while geodesics in $M$ are typically denoted by $\gamma$.
\end{convention}

\begin{lemma}\label{lemma:totgeo}
The flat planes distribution \eqref{eq:flatplanesdistr} is totally geodesic. 
\end{lemma}

\begin{proof}
Let $v \in S_pM$ and $w \in \mathcal Z_v$. The geodesic $c(t)=\cos(t) v+\sin(t) w$ on $S_pM$ satisfies $c(0)=v$ and $\dot{c}(0)=w$. It suffices to show that $\dot{c}(t)\in \mathcal Z_{c(t)}$ for all $t \in \R$. A direct calculation gives $\mathcal J_{c(t)}(\dot{c}(t))=-\sin(t)\mathcal J_w(v)+\cos(t)\mathcal J_v(w)$. Since $w\in \mathcal Z_v$, we have $\mathcal J_v(w)=0$. By Lemma~\ref{lemma:crit}, $\mathcal J_w(v)=0$. Therefore, $\mathcal J_{c(t)}(\dot{c}(t))=0$, concluding the proof.
\end{proof}
%\begin{remark}
%The above can also be deduced as follows.\marginpar{remove?} Consider the Stiefel manifold $V_2(T_pM)$ of ordered pairs of orthonormal vectors in $T_pM$, which is canonically identified with the unit tangent bundle of $S_pM$. The \emph{linear span} function gives $V_2(T_pM)$ the structure of a principal $\mathsf O(2)$-bundle over the Grassmannian of $2$-planes in $T_pM$. Given $(v,w)\in V_2(T_pM)$, Lemma~\ref{lemma:crit} implies that the property $w\in\mathcal Z_v$ depends only on the linear span $v\wedge w$, and is hence $\mathsf O(2)$-invariant on $V_2(T_pM)$. This corresponds precisely to \eqref{eq:flatplanesdistr} being totally geodesic, since the $\mathsf O(2)$-orbit of $(v,w)\in V_2(T_pM)$ descends to a geodesic via the bundle projection onto $S_pM$. Notice this is the geodesic given by the intersection of the linear subspace of $T_pM$ spanned by $v$ and $w$ with $S_pM$.
%\end{remark}

The possible flat planes distributions in our particular framework of $3$-manifolds with $\cvc(0)$ are easily classified. Namely, there are only two possible totally geodesic tangent distributions $\mathcal D$ on a round $2$-sphere such that $\dim\mathcal D_v\geq1$ for all $v\in S^2$, see Figure~\ref{fig:zeroplanesdistr}.
\begin{enumerate}
\item\label{item:1} Suppose there exist $v_1,v_2\in S^2$, with $v_1\neq\pm v_2$ and $\dim\mathcal D_{v_i}=2$ for $i=1,2$. Consider $v_3 \in S^2$ that does not lie in the great circle $C_{12}$ through $\{v_1,v_2\}$. Then $\mathcal D_{v_3}=T_{v_3} S^2$, since the two lines at $v_3$ tangent to the great circles $C_{13}$ and $C_{23}$ through $\{v_1,v_3\}$ and $\{v_2,v_3\}$, respectively, are transverse. Since no $v\in S^2$ lies simultaneously  in all three great circles $C_{12}$, $C_{13}$ and $C_{23}$, repeating the above argument it follows that $\mathcal D_v=T_v S^2$ for all $v\in S^2$.

\item\label{item:2} Suppose there are no $v_1,v_2\in S^2$ as above. Since any two great circles in $S^2$ intersect, there exists $\xi \in S^2$ with $\dim \mathcal D_{\xi}=2$. As $\mathcal D$ is totally geodesic, it follows that also $\dim \mathcal D_{-\xi}=2$. Each geodesic joining $\xi$ to $-\xi$ is tangent to $\mathcal D$. Thus, the tangent fields to these geodesics span $\mathcal D$ on $S^2 \setminus \{\pm \xi\}$.
\end{enumerate}

\begin{figure}[htf]
\centering
\includegraphics[scale=.6]{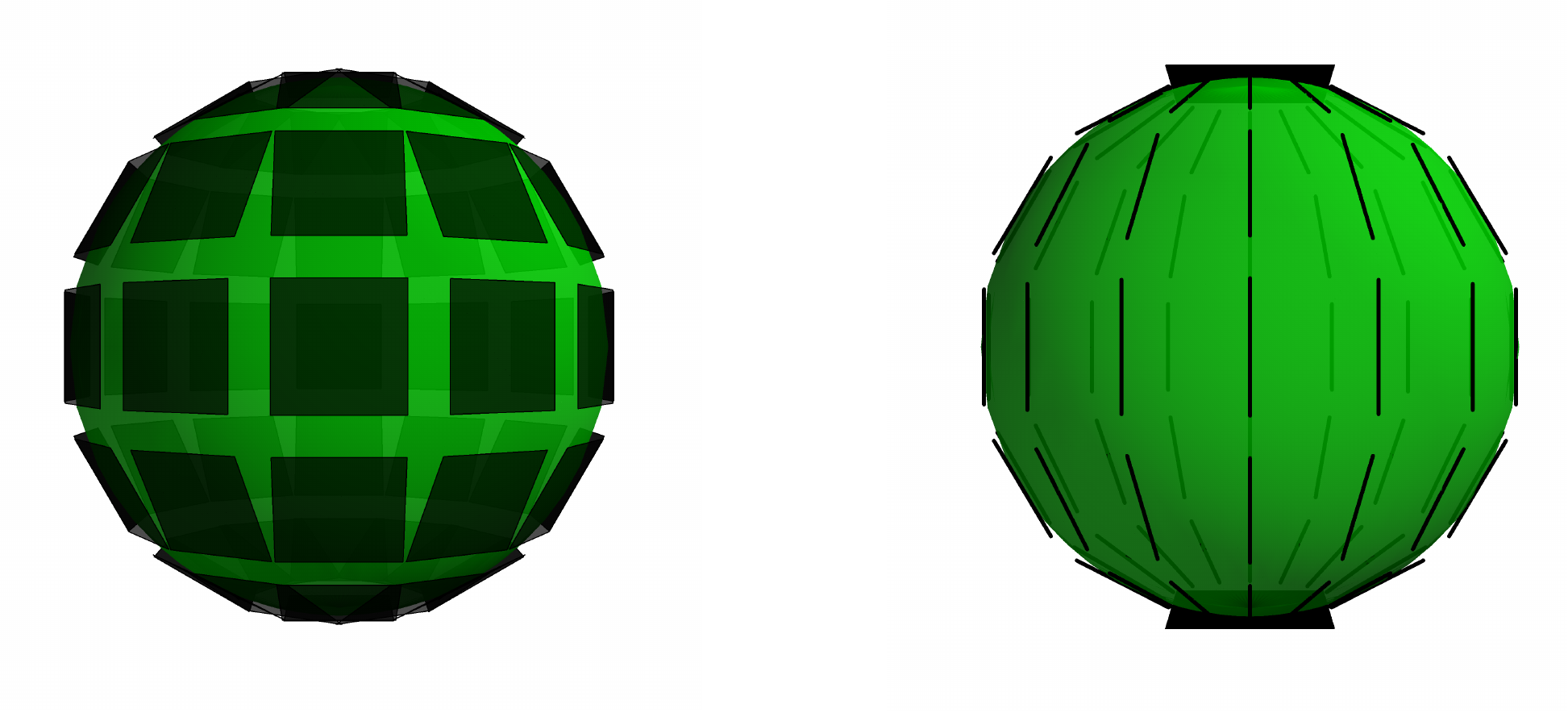}
\begin{pgfpicture}
\pgfputat{\pgfxy(-9.4,0.2)}{\pgfbox[center,center]{(1)}}
\pgfputat{\pgfxy(-2.5,0.2)}{\pgfbox[center,center]{(2)}}
\end{pgfpicture}
\caption{The two possible flat planes distributions on $S_pM$.}
\label{fig:zeroplanesdistr}
\end{figure}

\begin{definition}
A point $p \in M$ is called \emph{isotropic} if all $2$-planes in $T_pM$ have the same sectional curvature.\footnote{Note that, under the assumption that $M$ has $\cvc(0)$, \emph{isotropic} points are the points at which all tangent planes are flat.} Otherwise, $p$ is called a \emph{nonisotropic point}. We denote by $\mathcal{I}$ the closed subset of isotropic points in $M$, and by $\mathcal{O}=M \setminus \mathcal{I}$ the open subset of nonisotropic points.
\end{definition}

The above discussion, together with Lemma \ref{lemma:totgeo}, implies that the flat planes distribution at $p\in M$ is of the form \eqref{item:1} or \eqref{item:2} above, according to $p$ being an isotropic or nonisotropic point, respectively. More precisely, we have proved:

\begin{lemma}\label{lemma:cvc}
Let $M$ be a $3$-manifold with $\cvc(0)$ and pointwise signed sectional curvatures. Given $p \in M$, consider the flat planes distribution $\mathcal Z$ on $S_pM$.
\begin{enumerate}
\item\label{item:cvc0-1} The following are equivalent:
\begin{enumerate}
\item[(i)] $p \in \mathcal{I}$ is an isotropic point;
\item[(ii)] There exist $v_1,v_2\in S_pM \setminus \textgoth{Z}_p$, with $v_1\neq\pm v_2$;
\item[(iii)] $\mathcal Z=T S_pM$ is the tangent bundle of $S_pM$.
\end{enumerate}
\item\label{item:cvc0-2} If $p\in \mathcal{O}$ is a nonisotropic point, there exists a $1$-dimensional subspace $L_p$ of $T_pM$ such that $S_pM \setminus \textgoth Z_p=S_pM \cap L_p$. The tangent lines to great circles containing the singular set $S_pM \cap L_p$ define $\mathcal Z$ on the regular set $\textgoth Z_p$.
\end{enumerate}
\end{lemma}

\subsection{Nonisotropic components}
As before, let $M$ be a complete $3$-manifold with $\cvc(0)$ and pointwise signed sectional curvatures. Let $\mathcal {C}\subset\mathcal{O}$ be a (nonempty) connected component of nonisotropic points on $M$. From Lemma~\ref{lemma:cvc}, we have a well-defined line field $L$ on $\mathcal C$. The following summarizes structural results of Schmidt and Wolfson~\cite[Thm 1.3, Cor 2.10]{scwo} concerning this line field.

\begin{theorem}\label{thm:summary}
The line field $L$ on any connected component $\mathcal{C}$ of nonisotropic points on $M$ is smooth and tangent to a foliation of $\mathcal{C}$ by complete geodesics. Furthermore, if $\mathcal C$ has finite volume, then the line field $L$ is parallel on $\mathcal C$.
\end{theorem}

The finite volume assumption in the last statement can be omitted if further geometric information on the tangent distribution $L^\perp$ is available, as follows.

\begin{proposition}\label{prop:totgeoandparallel}
Let $M$ be a $3$-manifold as above, with possibly infinite volume. If the tangent distribution $L^\perp$ is totally geodesic on $\mathcal C$, then $L$ is parallel on $\mathcal C$.
\end{proposition}

\begin{proof}
Let $B\subset \mathcal{C}$ be a small metric ball, and let $V$ be a unit vector field such that $L_p=\spn\{V_p\}$ for all $p\in B$. It suffices to prove that $\nabla V=0$ on $B$, since any curve contained in $\mathcal{C}$ can be covered by finitely many such balls.
By Theorem~\ref{thm:summary}, $\nabla_V V=0$ on $B$.
Consider $\hat S_p\colon T_pM \to T_pM$, $\hat S_p(w)=-\nabla_{w} V$. Since $V$ has constant length one, $\hat S_p$ restricts to an operator 
\begin{equation}
S_p\colon L_p^{\perp} \to L_p^{\perp}, \quad S_p(w)=-\nabla_w V,
\end{equation}
on the subspace $L_p^\perp$ of vectors orthogonal to $V_p$, cf.\ \eqref{eq:shapeop}. We now use that $L^\perp$ is totally geodesic to show that $S_p$ must also vanish.\footnote{Notice that this is not immediate, since $L^\perp$ is \emph{a priori} not necessarily integrable.}
Let $\gamma(t)$ be an orbit of the flow generated by $V$, with $\gamma(0)=p$. According to \cite[Thm 2.9]{scwo}, the operators $S(t)=S_{\gamma(t)}\colon L_{\gamma(t)}^\perp\to L_{\gamma(t)}^\perp$ satisfy the evolution equation
\begin{equation*}
\tfrac{\dd}{\dd t} \tr S(t)=(\tr S(t))^2-2 \det S(t).
\end{equation*}
Let $\{e_1,e_2\}$ be an orthonormal $2$-frame along $\gamma(t)$ that spans $L^{\perp}_{\gamma(t)}$. It follows from \eqref{eq:totgeodD1} that $\tr S(t)=\g(S(t)(e_1),e_1)+\g(S(t)(e_2),e_2)=0$, for all $t$. Thus, using \eqref{eq:totgeodD},
\begin{equation*}
0=\det S(t)=-\g(S(t)(e_1),e_2)\g(S(t)(e_2),e_1)=\g(S(t)(e_1),e_2)^2=\g(S(t)(e_2),e_1)^2.
\end{equation*}
By the above, $S(t)$ vanishes identically, concluding the proof.
\end{proof}

\section{\texorpdfstring{Structure of $3$-manifolds with higher rank}{Structure of 3-manifolds with higher rank}}\label{sec:structure}

Let $(M,\g)$ be a complete $3$-manifold. For each $p \in M$ and $v \in S_pM$, denote by
\begin{equation}\label{eq:paralleltransl}
P_t\colon T_pM \to T_{\gamma_v(t)} M
\end{equation}
the linear isomorphisms defined by parallel translation along $\gamma_v(t)$, and define
\begin{equation}\label{eq:eucldistr}
\mathcal R_v:=\big\{w \in v^{\perp} : P_t(w) \text{ is a parallel Jacobi field along } \gamma_v(t) \big\}.
\end{equation}
The \emph{rank} of $v\in S_pM$ is defined to be $\dim \mathcal R_v+1$. Similarly, the \emph{rank} of a line $L$ in $T_pM$ is defined as the rank of a unit vector tangent to $L$, and the \emph{rank} of a geodesic is the (common) rank of its unit tangent vectors. 
The manifold $(M,\g)$ is said to have \emph{higher rank} if $\dim \mathcal R_v\geq1$ for all $v\in S_pM$ and all $p\in M$.
Throughout the remainder of this section, $(M,\g)$ will denote a complete higher rank $3$-manifold.  

\subsection{Curvature}
We begin by analyzing the curvature of higher rank $3$-manifolds.

\begin{proposition}\label{prop:signed}
A $3$-manifold $(M,\g)$ with higher rank has pointwise signed sectional curvatures; i.e., for all $p\in M$ either $\sec_p\geq0$ or $\sec_p\leq0$.
\end{proposition}

\begin{proof}
Let $\{v_i\}$ be a local orthonormal frame near $p \in M$ that diagonalizes the Ricci tensor. Since $\dim M=3$, the curvature operator\footnote{According to \eqref{eq:R}, the curvature operator is defined by $\langle R(X\wedge Y),Z\wedge W\rangle=R(X,Y,W,Z)$.} $R\colon\wedge^2 T_pM\to\wedge^2 T_pM$ decomposes as $R=(\Ric-\frac{\scal}{4}\g)\owedge\g$, where $\owedge$ denotes the Kulkarni-Nomizu product. In particular, $R$ is diagonalized by the orthonormal basis $\{v_i\wedge v_j\}$, and hence $\mathcal{J}_{v_i}(v_j)=R(v_j,v_i)v_i=\sec(v_i \wedge v_j)v_j$ for each $i,j \in \{1,2,3\}$. Let $\sigma\subset T_pM$ be a $2$-plane orthogonal to the unit vector $\sum_{i=1}^{3} c_i v_i\in T_pM$. Then $\sigma=*\textstyle\sum_i c_i v_i=\textstyle\sum_i c_i (v_{i+1}\wedge v_{i+2})$, where indices are modulo $3$. Therefore (cf.\ \cite[Lemma 2.2]{scwo}),
\begin{equation*}
\sec(\sigma)=\langle R(\sigma),\sigma\rangle=c_1^2 \sec(v_2 \wedge v_3)+c_2^2\sec(v_3 \wedge v_1)+c_3^2\sec(v_1 \wedge v_2).
\end{equation*}
To prove the statement, it suffices to show that at least two of the three sectional curvatures $\sec(v_i\wedge v_j)$ vanish. If this were not the case, after possibly reindexing, we would have that $\sec(v_1 \wedge v_2)$ and $\sec(v_1 \wedge v_3)$ are both nonzero. Since $M$ has higher rank, there exists a unit vector $w\in T_pM$ perpendicular to $v_1$ that is the initial condition for a normal parallel Jacobi field along the geodesic $\gamma_{v_1}$. It follows from the Jacobi equation that $\mathcal{J}_{v_1}(w)=R(w,v_1)v_1=0$. Thus,
\begin{equation*}
0=\mathcal{J}_{v_1}(w)=\g(w,v_2)\sec(v_1\wedge v_2)\,v_2+\g(w,v_3)\sec(v_1 \wedge v_3)\,v_3.
\end{equation*}
As both $\sec(v_1 \wedge v_2)$ and $\sec(v_1 \wedge v_3)$ are nonzero, this implies that $\g(w,v_i)=0$ for $i=2,3$, hence $w=0$, a contradiction.
\end{proof}

\subsection{Rank distribution}
As $3$-manifolds with higher rank have pointwise signed sectional curvatures, it follows from Lemma~\ref{lemma:crit} that the space $\mathcal R_v$ defined in \eqref{eq:eucldistr} coincides with the space of $w\in v^\perp$ such that $\sec\big(\dot{\gamma_v}(t) \wedge P_t(w)\big)=0$ for all $t\in\R$. Moreover, the discussion of $3$-manifolds with $\cvc(0)$ and signed curvatures in Section~\ref{sec:cvc0} automatically applies to $3$-manifolds of higher rank. In particular, $\mathcal R_v$ is a linear subspace of $\mathcal Z_v$, see Definition~\ref{def:flatplanesdistr}. This motivates the following:

\begin{definition}\label{def:rankdistr}
The \emph{rank distribution} $\mathcal R$ is the tangent distribution on $S_pM$ given by the association $v \mapsto\mathcal R_v$, see \eqref{eq:eucldistr}. We denote by $\textgoth R_p$ its regular set, which is the open subset of $S_pM$ formed by rank $2$ vectors.
\end{definition}

\begin{lemma}\label{lemma:cont}
The restriction of  $\mathcal R$ to $\textgoth{R}_p$ is continuous for each $p \in M$.
\end{lemma}

\begin{proof}
Let $v_i \in \textgoth{R}_p$ be a sequence converging to $v \in \textgoth{R}_p$.  Assume $w \in T_v (S_pM)$ is an accumulation point of a sequence of unit vectors $w_i \in\mathcal R_{v_i}$. Then $w \in\mathcal R_v$, by continuity of sectional curvatures and parallel translation. Since $v$ has rank $2$, $\mathcal R_{v}=\spn \{w\}$, and hence the lines $\mathcal R_{v_i}=\spn\{w_i\}$ converge to $\mathcal R_{v}=\spn\{w\}$.
\end{proof}

The facts proved for the flat planes distribution in Section~\ref{sec:cvc0} yield the following.

\begin{lemma}\label{lemma:oneline}
If there are two or more lines in $T_pM$ of rank $3$, then $p$ is isotropic.
\end{lemma}

\begin{proof}
Follows directly from Lemma \ref{lemma:cvc} \eqref{item:cvc0-1}, since $\mathcal R_{v}$ is a subspace of $\mathcal Z_{v}$.
\end{proof}

\begin{lemma}\label{lemma:same}
The flat planes distribution $\mathcal Z$ and the rank distribution $\mathcal R$ coincide on $S_pM$, provided $p \in \mathcal{O}$ is a nonisotropic point.
\end{lemma}

\begin{proof}
Since $S_pM$ is a $2$-sphere, it does not admit a continuous tangent line field, so Lemma \ref{lemma:cont} implies $\dim \mathcal R_v=2$ for some $v \in S_pM$. As $\mathcal R_w$ is a subspace of $\mathcal Z_w$ for each $w \in S_pM$, Lemmas \ref{lemma:cvc} \eqref{item:cvc0-2} and \ref{lemma:oneline} imply that $v \in L_p$ and hence $\mathcal R_w=\mathcal Z_w$.
\end{proof}

\begin{lemma}\label{lemma:rankonevectors}
The restriction of $\mathcal R$ to $\textgoth{R}_p$ is smooth for each $p \in M$.
\end{lemma}

\begin{proof}
Let $B \subset \textgoth{R}_p$ be a metric ball with center $b_0 \in B$. It suffices to prove smoothness of $\mathcal R$ on $B$. As $b_0$ is a vector of rank $2$, there exist a unit vector $w\in b_{0}^{\perp}$ and $t>0$, such that $\sec(\dot{\gamma}_{b_0}(t)\wedge P_{t}(w))\neq 0$.  In particular $\gamma_{b_0}(t) \in \mathcal{O}$.  As $\mathcal{O}$ is open, we may assume $\gamma_{b}(t) \in \mathcal{O}$ for all $b \in B$ up to shrinking $B$. From Lemma~\ref{lemma:same}, $\mathcal R_{\dot{\gamma}_b(t)}=\mathcal Z_{\dot{\gamma}_b(t)}$ for each $b \in B$. As the unit tangent vectors $\dot{\gamma}_b(t)$ vary smoothly with $b \in B$, also $\mathcal R_{\dot{\gamma}_b(t)}$ varies smoothly with $b\in B$, see Remark~\ref{rem:chi}. To conclude the proof, notice $\mathcal R_b$ is obtained by parallel translating for time $t$ along $\gamma_b$ the line $\mathcal R_{\dot{\gamma}_b(t)}$ to $T_b (S_pM)$.
\end{proof}

\subsection{Adapted frame}\label{adapted}
Assume $\textgoth{R}_p\neq\emptyset$ and let $B$ be a metric ball contained in $\textgoth{R}_p\subset S_pM$. By Lemma~\ref{lemma:rankonevectors}, there is a smooth unit vector field $e_1$ on $B$ tangent to the rank distribution $\mathcal R$. An orientation on $S_pM$ determines an orthonormal frame $\{e_1,e_2\}$ on $B$. For each $v \in B$ and $t>0$, define
\begin{equation}
E_0(t)=P_t(v)=\dot{\gamma}_v(t), \quad \text{ and } \quad E_i(t)=P_t(e_i), \; i=1,2.
\end{equation}
Then $\{E_0(t),E_1(t),E_2(t)\}$ is a parallel orthonormal frame along $\gamma_v$. By construction, we have that $E_1(t)$ \emph{is a parallel Jacobi field along} $\gamma_v(t)$, in accordance with the higher rank assumption.\footnote{An analogous construction in the weaker context of a finite-volume $\cvc(0)$ $3$-manifold is possible, by choosing the vector field $e_1$ to be tangent to the flat planes distribution $\mathcal Z_v$ near $v\in\textgoth{Z}_p$, see Theorem~\ref{thm:summary}. Nevertheless, in this case, one only concludes that $E_1(t)$ is tangent to the line field $L_{\gamma(t)}$ for $t$ such that $\gamma(t)$ is in the same component $\mathcal C$ of nonisotropic points as $\gamma(0)$. In particular, $\sec(E_0(t)\wedge E_1(t))$ might not vanish if $\gamma(t)\not\in\mathcal C$. This is a crucial step where the higher rank assumption (as opposed to $\cvc(0)$) is used in our main result.} In particular, $\sec(E_0(t)\wedge E_1(t))=0$ for all $t\geq0$.

Fix $v \in B$ and $t_0>0$ for which $t_0 v$ is not a critical point of $\exp_p\colon T_pM\to M$, so that $\exp_p$ is a diffeomorphism between a neighborhood $U$ of $t_0 v$ and a neighborhood $V$ of $\exp_p(t_0v)$. Up to shrinking $B$, we may assume that the image of $U$ under the radial retraction $T_pM\setminus\{0\}\to S_pM$ coincides with $B$. The adapted frames along geodesics issuing from $p\in M$ with initial velocity in $B$ provide an orthonormal \emph{adapted frame} $\{E_0,E_1,E_2\}$ of the open set $V$ in $M$. 
For $t$ near $t_0$, the distance spheres $S(p,t)=\{x\in M:\dist(p,x)=t\}$ intersect the  neighborhood $V$ of $\exp_p(t_0v)$ in smooth codimension one submanifolds, see Figure~\ref{fig:coord}.
The vector fields $E_1(t)$ and $E_2(t)$ are tangent to $S(p,t)$ in $V$, and its outward pointing unit normal is $E_0(t)$.

% PICTURE OF COORDINATES -- PROBABLY UNNECESSARY?
\begin{figure}[htf]
\centering
\vspace{-0.1cm}
\includegraphics[scale=0.75]{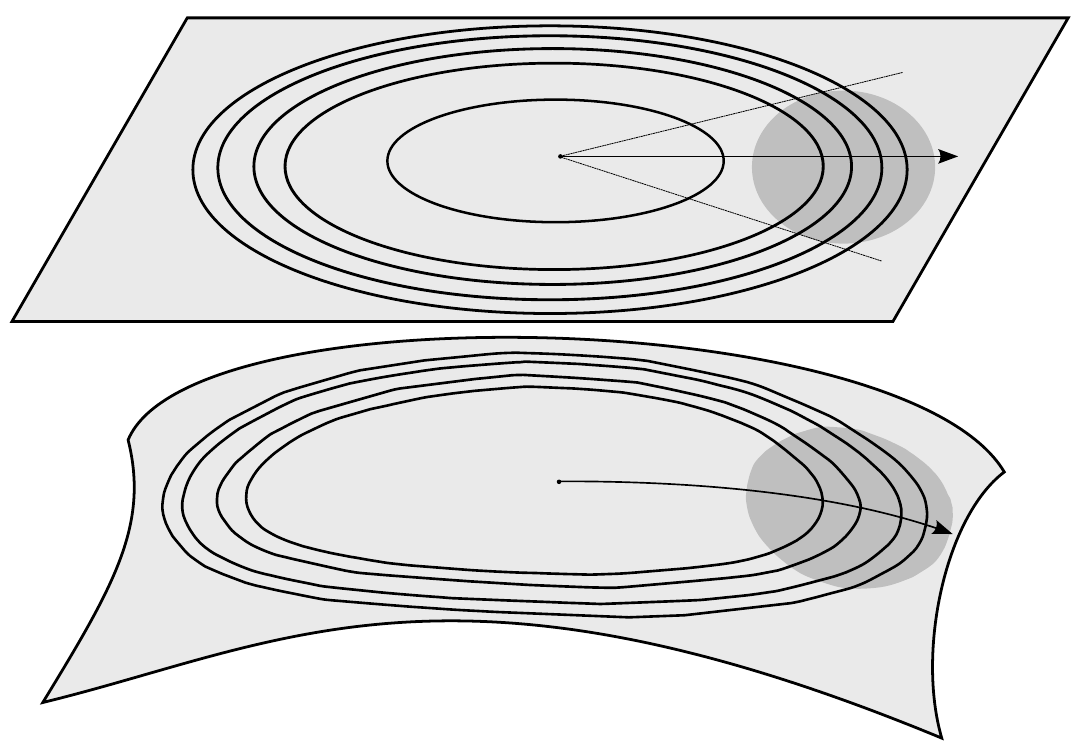}
\begin{pgfpicture}
\pgfputat{\pgfxy(-1.2,4.9)}{\pgfbox[center,center]{$U$}}
\pgfputat{\pgfxy(-2.72,4.65)}{\pgfbox[center,center]{$v$}}
\pgfputat{\pgfxy(-4.1,1.8)}{\pgfbox[center,center]{$p$}}
\pgfputat{\pgfxy(-1.2,2.3)}{\pgfbox[center,center]{$V$}}
\pgfputat{\pgfxy(-0.7,1.5)}{\pgfbox[center,center]{$\gamma_v$}}
\pgfputat{\pgfxy(-8.3,5)}{\pgfbox[center,center]{$T_pM$}}
\pgfputat{\pgfxy(-8.3,1.5)}{\pgfbox[center,center]{$M$}}
\end{pgfpicture}
\vspace{-0.2cm}
\caption{The neighborhood $V$ of $\exp_p(t_0v)$, where the adapted frame $\{E_0,E_1,E_2\}$ is defined.}
\label{fig:coord}
\end{figure}

We now compute the Christoffel symbols of the adapted frame $\{E_0,E_1,E_2\}$. Let $a_{12}^1$ and $a_{12}^2$ be smooth functions on $B$ such that 
$[e_1,e_2]=a_{12}^1 e_1+a_{12}^2 e_2$.
For each geodesic $\gamma_v$ with initial velocity $v\in B$, let $J_i(t)$, $i=1,2$, denote the Jacobi field along $\gamma_v$ with initial conditions $J_i(0)=0$ and $J_{i}'(0)=e_i\in T_v(S_pM)$.  For $t$ near $t_0$, consider $F_t\colon B \to M$ given by $F_t(v)=\exp_p(tv)$. From \eqref{jabfield}, we have $\dd F_t(e_i)=J_i(t)$, $i=1,2$. Using the above and the fact that $J_i$ are invariant under the radial flow generated by $E_0$, one obtains:
\begin{align}\label{brackets1}
&[J_1,J_2]=a_{12}^1J_1+a_{12}^2J_2,&
&[E_0,J_1]=\mathcal{L}_{E_0} J_1=0,&      
&[E_0,J_2]=\mathcal{L}_{E_0} J_2=0.&
\end{align}

\begin{lemma}\label{lemma:jacobi}
For $v \in B$, we have that $J_1(t)=tE_1(t)$ and $J_2(t)=f(t)E_2(t)$, where $f(t)$ is the solution of the ODE
\begin{equation}\label{eq:ode-f}
f''+\sec(E_0\wedge E_2)f=0, \quad \text{ with }\quad f(0)=0, \, f'(0)=1.
\end{equation}
\end{lemma}

\begin{proof}
From $\sec(E_0\wedge E_1)=0$ and Lemma \ref{lemma:crit}, we have $R(E_1,E_0)E_0=0$. Since $E_1(t)$ is parallel, the field $tE_1(t)$ is a Jacobi field. Moreover, it has the same initial conditions as $J_1(t)$ and therefore $J_1(t)=tE_1(t)$. Regarding $J_2(t)$, there exist smooth functions $h(t)$ and $f(t)$ so that $J_2(t)=h(t)E_1(t) +f(t)E_2(t)$. Thus, the Jacobi equation reads
\begin{equation}\label{eq:jacobiaux}
0=J_2''+R(J_2,E_0)E_0=h''E_1+f''E_2+fR(E_2,E_0)E_0.
\end{equation}
Since $E_1$ is an eigenvector of $\mathcal J_{E_0}$, so is $E_2$. Consequently, $R(E_2,E_0,E_0,E_1)=0$. Thus, taking the inner product of \eqref{eq:jacobiaux} with $E_1$, the above gives $h''=0$ and hence $h(t)$ is linear. The initial conditions $J_2(0)=0$ and $J_2'(0)=e_2$ respectively imply $h(0)=0$ and $h'(0)=0$, so $h\equiv 0$.
\end{proof}

Altogether, Lemma \ref{lemma:jacobi} and \eqref{brackets1} imply that the Lie brackets of $\{E_1,E_2,E_3\}$ are
\begin{equation*}
[E_0, E_1]= -\tfrac{1}{t} E_1, \quad [E_0,E_2] =-\tfrac{f'}{f} E_2, \quad [E_1,E_2]=\tfrac{a_{12}^1}{f} E_1-\left(\tfrac1f E_1(f)-\tfrac{a_{12}^2}{t}\right)E_2,
\end{equation*}
and hence applying the Koszul formula \eqref{eq:koszul} we have the following:

\begin{lemma}\label{lemma:christoffel1}
The adapted frame $\{E_0,E_1,E_2\}$ has Christoffel symbols given by:
\begin{align*}
&\nabla_{E_0} {E_0}=0,& 	&\nabla_{E_1} E_0=\tfrac1tE_1,&			 	&\nabla_{E_2} E_0=\tfrac{f'}{f}E_2,\\
&\nabla_{E_0} E_1=0,&	 &\nabla_{E_1} E_1=-\tfrac1t E_0-\tfrac{a_{12}^1}{f}E_2,& 	&\nabla_{E_2} E_1=\left(\tfrac1f E_1(f)-\tfrac{a_{12}^2}{t}\right)E_2, \\
&\nabla_{E_0} E_2=0,& 	 &\nabla_{E_1} E_2=\tfrac{a_{12}^1}{f}E_1,& 		&\nabla_{E_2} E_2=-\tfrac{f'}{f}E_0-\left(\tfrac1f E_1(f)-\tfrac{a_{12}^2}{t}\right)E_1,
\end{align*}
where $f(t)$ is the solution of the ODE \eqref{eq:ode-f}.
\end{lemma}

\begin{lemma}\label{lemma:a121}
The vector field $e_1$ is geodesic on $B$ if and only if $a_{12}^1\equiv0$.
\end{lemma}

\begin{proof}
Denote by $\langle \cdot, \cdot \rangle$ the metric on $S_pM$ and by $\overline{\nabla}$ its Levi-Civita connection. As the vector field $e_1$ has unit length on $B$, we have $\langle \overline{\nabla}_{e_1} e_1,e_1\rangle=\tfrac12e_1\langle e_1,e_1\rangle=0$. To conclude, notice that the Koszul formula \eqref{eq:koszul} yields
\begin{equation*}
\langle \overline{\nabla}_{e_1} e_1,e_2\rangle=\tfrac12\big(\langle [e_1,e_1],e_2\rangle- \langle[e_1,e_2],e_1\rangle+\langle [e_2,e_1],e_1\rangle\big)=-a_{12}^1.\qedhere
\end{equation*}
\end{proof}

We now verify that $e_1$ is a geodesic vector field on $B$. This is easily deduced when $p \in \mathcal{O}$, since by Lemma~\ref{lemma:same} we have $\mathcal R=\mathcal Z$, which is totally geodesic by Lemma \ref{lemma:totgeo}. Suppose now $p\notin\mathcal O$ and $\textgoth{R}_p\neq\emptyset$, in which case we have:

\begin{proposition}\label{prop:totgeo-iso}
The restriction of $\mathcal R$ to $\textgoth{R}_p$ is a totally geodesic distribution.
\end{proposition}

\begin{proof}
By Lemma \ref{lemma:a121}, it suffices to show that $a_{12}^1\equiv0$ on $B$, since $e_1$ is tangent to $\mathcal R$. If this were not the case, we may assume $a_{12}^{1}(b)\neq 0$ for all $b\in B$, up to shrinking $B$. Since $\sec(E_0\wedge E_1)=0$, Lemma \ref{lemma:crit} implies that $R(E_0,E_1)E_1=0$. In particular, $R(E_1,E_2,E_0,E_1)=0$. On the other hand, from Lemma~\ref{lemma:christoffel1}, we have:
\begin{equation*}
R(E_1,E_2,E_0,E_1)= a_{12}^1\frac{f't-f}{f^2 t}.
\end{equation*}
Since we assumed $a_{12}^1$ is nonzero on $B$, it follows that $f' t-f=0$ along every geodesic $\gamma_v(t)$ with initial velocity $v \in B$, for all $t>0$ such that $tv$ is not a critical point of $\exp_p\colon T_pM\to M$.\footnote{i.e., the $t>0$ for which Lemma~\ref{lemma:christoffel1} is valid.} The Jacobi fields $J_1(t)=tE_1(t)$ and $J_2(t)=f(t)E_2(t)$ form a basis of the initially vanishing normal Jacobi fields along $\gamma_v(t)$, see Lemma~\ref{lemma:jacobi}. Thus, $tv$ is not a critical point of $\exp_p$ for all $t>0$ such that $f(t)\neq 0$. In particular, $f't-f=0$ for sufficiently small $t>0$. Differentiating with respect to $t$ yields $f''t=0$, hence $f''=0$. As $f(0)=0$ and $f'(0)=1$, we have $f(t)=t$, and hence $tv$ is not a critical point of $\exp_p$ for all $t>0$. Moreover, \eqref{eq:ode-f} implies $\sec(E_0 \wedge E_2)=0$ for all $t>0$, and hence $v$ is a rank $3$ vector, contradicting $v\in B\subset \textgoth R_p$.
\end{proof}

\begin{remark}\label{rem:simpleChristoffel}
By Proposition~\ref{prop:totgeo-iso}, the Christoffel symbols of $\{E_0,E_1,E_2\}$ given in Lemma~\ref{lemma:christoffel1} can be \emph{a posteriori} simplified using that $a_{12}^1\equiv0$ on $B$.
\end{remark}

\begin{corollary}\label{cor:nocircles}
A great circle $C\subset S_pM$ that is everywhere tangent to the rank distribution $\mathcal{R}$ cannot consist entirely of rank two vectors. 
%The subset $\textgoth{R}_p$ of rank $2$ vectors does not contain any great circle.
\end{corollary}

\begin{proof}
We argue by contradiction.  Asssume that some great circle $C \subset S_pM$ consists entirely of rank $2$ vectors and is everywhere tangent to the rank distribution $\mathcal{R}$. As the set of rank two vectors $\textgoth{R}_p$ is open in $S_pM$, some tubular neighborhood $N$ of $C$ in $S_pM$ consists entirely of rank $2$ vectors. Orient the rank line field $\mathcal R$ on $N$ with a unit length vector field $e_1$, and denote by $\phi_t$ the local flow generated by $e_1$. For $v\in N$ sufficiently close to $C$, the orbit $\phi_t (v)$ remains in $N$ for all $t \in [0,2\pi]$. Proposition~\ref{prop:totgeo-iso} implies that this orbit is a great circle $\overline{C}$ of $S_pM$. The two great circles $C$ and $\overline{C}$ must intersect transversally at some point $x$, where the tangent lines $T_x C$ and $T_x \overline{C}$ are both subspaces of $\mathcal R_x$. Therefore $x$ is a rank $3$ vector of $S_pM$, a contradiction.
\end{proof}

\subsection{Totally geodesic flats}\label{subsec:flats}
Recall that if $p\in\mathcal O$ is a nonisotropic point, then there exists a unique rank $3$ line $L_p$ in $T_pM$, see Lemma~\ref{lemma:oneline}. We now show that the linear open book decomposition of $T_pM$ with binding $L_p$ and pages given by $2$-planes that contain $L_p$ exponentiates to the open book decomposition of domains $U_p=M\setminus\cut(p)$ of the form \eqref{eq:injball} mentioned in the Introduction.

\begin{proposition}\label{prop:flats}
If $p\in\mathcal O$ and $\sigma$ is a $2$-plane in $T_pM$ containing $L_p$, then $\exp_p\colon\sigma\to M$ is an isometric immersion with image
$\Sigma:=\exp_p(\sigma)$ a totally geodesic flat immersed submanifold of $M$.
\end{proposition}

\begin{proof}
Let $\{v,w\}$ be an orthonormal basis of $\sigma$, and use parallel translation along $t\mapsto tv$ to identify $\{v,w\}$ with an orthonormal basis of $T_{tv} \sigma$, $t>0$. Let $J(t)$ be the normal Jacobi field along $\gamma_{v}(t)$ with $J(0)=0$ and $J'(0)=w$. Lemmas~\ref{lemma:cvc} and \ref{lemma:same} imply that $w \in \mathcal R_v$. Lemma~\ref{lemma:jacobi} implies that $J(t)=tP_t(w)$, where $P_t$ is given by \eqref{eq:paralleltransl}. Using \eqref{jabfield}, we have:
\begin{equation}\label{eq:E0E1}
\begin{aligned}
\dd(\exp_p)_{tv}(v)&=\dot{\gamma}_v(t)=P_t(v), \\
\dd(\exp_p)_{tv}(w)&=\tfrac{1}{t}\dd(\exp_p)_{tv}(tw)=\tfrac{1}{t} J(t)=P_t(w).
\end{aligned}
\end{equation}
Since $P_t$ is a linear isometry, it follows that $\exp_p\colon\sigma\to\Sigma$ is an isometric immersion.

The rank $2$ unit vectors $v \in\sigma\setminus L_p$ determine an adapted frame $\{E_0,E_1,E_2\}$ along the restriction of $\exp_p$ to $\sigma \setminus L_p$. In this adapted frame, $E_2$ is a unit normal field along $\Sigma$. From Lemma~\ref{lemma:christoffel1} and Remark~\ref{rem:simpleChristoffel}, we have $\nabla_{E_0} E_2=\nabla_{E_1} E_2=0$. Thus, if $v$ is a rank $2$ vector and $tv$ is not a critical point of $\exp_p$, the second fundamental form of $\Sigma$ vanishes at $tv$. Since the subset of critical points of $\exp_p$ in $\sigma\setminus L_p$ has dense complement in $\sigma$, it follows that the second fundamental form of $\Sigma$ vanishes identically.
\end{proof}

\begin{remark}
In the above notation, consider $v \in \sigma$ a rank $2$ unit vector and $w \in \sigma$ a unit vector orthogonal to $v$. Then $w \in \mathcal R_v$ determines an adapted frame $\{E_0,E_1,E_2\}$ along $\gamma_v(t)$ and \eqref{eq:E0E1} becomes
\begin{equation}\label{eq:E00E01}
\begin{aligned}
\dd(\exp_p)_{tv}(v)&=E_0(t), \\
\dd(\exp_p)_{tv}(w)&=E_1(t).
\end{aligned}
\end{equation}
\end{remark}

\subsection{Parallel line field}\label{subsec:parallel}
We now use the above adapted frames to construct a parallel line field $X^p$ on domains $U_p=M\setminus\cut(p)$, for $p\in\mathcal O$. Let $\xi \in L_p$ be a unit vector and $v \in S_pM \cap \xi^{\perp}$. Consider the spherical geodesic segment
\begin{equation}\label{eq:cs}
c(s)=\cos(s)\,v+\sin(s)\,\xi, \quad s\in\left[-\tfrac\pi2,\tfrac\pi2\right],
\end{equation}
that joins $-\xi$ to $\xi$ and passes through $v$ when $s=0$. Let $\{e_1,e_2\}$ be an orthonormal frame on $S_pM \setminus \{\pm \xi\}$, with $e_1$ tangent to the rank distribution $\mathcal R$ and oriented by $e_1(c(s))=\dot{c}(s)$. The frame $\{e_1,e_2\}$ is rotationally invariant and induces an adapted frame $\{E_0,E_1,E_2\}$ on $U_p \setminus \gamma_{\xi}(\R)$ with Christoffel symbols given by Lemma~\ref{lemma:christoffel1} (see also Remark~\ref{rem:simpleChristoffel}).

\begin{lemma}\label{lemma:pde}
Let $a,b\colon U_p \setminus \gamma_{\xi}(\R)\to\R$ be smooth functions. The vector field $V=aE_0+bE_1$ is parallel if and only if $a$ and $b$ satisfy the following equations:
\begin{align}
&a'=0,& &b'=0,&\label{radial}\\
&E_2(a)=0,& &E_2(b)=0,&\label{rotational}\\
&E_1(a)=\tfrac{b}{t},& &E_1(b)=-\tfrac{a}{t},&\label{trig}\\
&a\tfrac{f'}{f}+b\left(\tfrac1f E_1(f)-\tfrac{a_{12}^2}{t}\right)=0.\label{trans}
\end{align}
\end{lemma}

\begin{proof}
From Lemma~\ref{lemma:christoffel1} and Remark~\ref{rem:simpleChristoffel}, we have:
\begin{align*}
\nabla_{E_0} V&=a'E_0+b'E_1,\\
\nabla_{E_1} V&= \left(E_1(a)-\tfrac{b}{t}\right)E_0+\left(E_1(b)+\tfrac{a}{t}\right)E_1,\\
\nabla_{E_2} V&= E_2(a) E_0+ E_2(b)E_1+\left(a\tfrac{f'}{f}+b\left(\tfrac1f E_1(f)-\tfrac{a_{12}^2}{t}\right)\right)E_2.\qedhere
\end{align*}
\end{proof}

\begin{lemma}\label{lemma:reduce}
Smooth functions $\bar{a},\bar{b}\colon S_pM \setminus \{\pm \xi\}\to\R$ satisfying
\begin{align}
&e_1(\bar{a})=\bar{b},&  &e_1(\bar{b})=-\bar{a},\label{reduce1}\\
&e_2(\bar{a})=0,& &e_2(\bar{b})=0,\label{reduce2} 
\end{align}
determine smooth functions $a,b\colon U_p \setminus \gamma_{\xi}(\R)\to\R$ satisfying \eqref{radial}, \eqref{rotational} and \eqref{trig}.
\end{lemma}

\begin{proof}
Assume that $\bar{a}$ and $\bar{b}$ satisfy the above and let $\mu\colon S_pM \to (0, \infty]$ denote the \emph{cut time} function \eqref{eq:cut}. For each $v \in S_pM \setminus \{\pm \xi\}$, define $a$ and $b$ so that
\begin{equation}\label{eq:defaabb}
a\big(\gamma_v(t)\big)=\bar{a}(v) \quad \text{ and } \quad b\big(\gamma_v(t)\big)=\bar{b}(v), \quad \text{ for all } t \in (0, \mu(v)).
\end{equation}
Note that $a$ and $b$ satisfy \eqref{radial} by construction.  Recall that the radial flow generated by $E_0$ carries $e_1,e_2 \in T_v (S_pM)$ respectively to the Jacobi fields $J_1(t)=tE_1(t)$ and $J_2(t)=f(t)E_2(t)$ along $\gamma_v(t)$. Thus, we have:
\begin{align*}
&e_1(\bar{a})=J_1(t)(a)=tE_1(a),& &e_1(\bar{b})=J_1(t)(b)=tE_1(b), \\
&e_2(\bar{a})=J_2(t)(a)=f(t)E_2(t)(a),& &e_2(\bar{b})=J_2(t)(b)=f(t)E_2(t)(b).
\end{align*}
Therefore, \eqref{reduce1} implies \eqref{trig}, and as $f \neq 0$ on $U_p$, \eqref{reduce2} implies \eqref{rotational}.
\end{proof}

\begin{lemma}\label{isotoo}
Smooth functions $a,b\colon U_p \setminus \gamma_{\xi}(\R)\to\R$ that satisfy \eqref{radial} on $U_p\setminus\gamma_\xi(\R)$ and \eqref{trans} on $\big(U_p \setminus \gamma_{\xi}(\R)\big)\cap\mathcal{O}$ also satisfy \eqref{trans} on $U_p \setminus \gamma_{\xi}(\R)$.
\end{lemma}

\begin{proof}
As $a$ and $b$ are smooth functions, they satisfy (\ref{trans}) on $\big(U_p \setminus \gamma_{\xi}(\R)\big)\cap \overline{\mathcal{O}}$. It remains to show that \eqref{trans} is satisfied on the interior points of $\big(U_p\setminus \gamma_{\xi}(\R)\big)\cap\mathcal{I}$. Assume $x$ is one such interior point. %Theorem~\ref{thm:summary} implies that $\gamma_{\xi}(\R) \subset \mathcal{O}$. 
There exists a vector $v \in S_pM \setminus \{\pm \xi\}$ and $0<t_0<t_1<t_2\leq+\infty$, with
\begin{equation*}
\gamma_v(t_1)=x, \quad \gamma_v\big([t_0,t_2]\big)\subset \mathcal{I},\quad \text{ and }\quad \gamma_v(t_0),\gamma_v(t_2) \in \overline{\mathcal{O}}.
\end{equation*}
Let $\{E_0(t),E_1(t),E_2(t)\}$ be the adapted frame along $\gamma_v(t)$. The curvature tensor vanishes identically on $\mathcal{I}$, hence $R(E_0,E_2,E_2,E_1)=0$ for $t \in [t_0,t_2]$. From Lemma~\ref{lemma:christoffel1} and Remark~\ref{rem:simpleChristoffel}, we have
\begin{equation*}
f\,R(E_0,E_2,E_2,E_1)=\left(f \left(\tfrac{a_{12}^{2}}{t}-\tfrac{1}{f}E_1(f)\right)\right)' \quad \text{ along } \gamma_v(t).
\end{equation*}
Thus, $f \left(\tfrac{a_{12}^{2}}{t}-\tfrac{1}{f}E_1(f)\right)$ is constant on $\gamma_v\big([t_0,t_2]\big)$. By \eqref{radial}, the functions $a$ and $b$ are also constant on $\gamma_v\big([t_0,t_2]\big)$. Therefore,
\begin{equation*}
\left(af'+bf\left(\tfrac{1}{f}E_1(f)-\tfrac{a_{12}^2}{t}\right)\right)'=af'', \quad \text{ for all }t \in [t_0,t_2].
\end{equation*}
As $\gamma_v(t) \in \mathcal{I}$ when $t \in [t_0,t_2]$, on this interval $\sec(E_0\wedge E_2)=-\tfrac{f''}{f}=0$. Thus,
\begin{equation}\label{eq:ugly}
af'+bf\left(\tfrac{1}{f}E_1(f)-\tfrac{a_{12}^2}{t}\right) \quad \text{ is constant on } [t_0,t_2].
\end{equation}
As $\gamma_v(t_0) \in \overline{\mathcal{O}}$, \eqref{trans} is satisfied when $t=t_0$. Equivalently, \eqref{eq:ugly} vanishes when $t=t_0$. Therefore, \eqref{trans} is satisfied for all $t \in [t_0,t_2]$, in particular, at $\gamma(t_1)=x$.
\end{proof}

We are now ready to construct the parallel line field $X^p$ on $U_p=M\setminus\cut(p)$, for $p\in\mathcal O$. As before, let $\mu\colon S_pM \to (0,\infty]$ denote the cut time function \eqref{eq:cut}. For each $x \in U_p \setminus \{p\}$, there are unique $w_x \in S_pM$ and $t_x \in (0, \mu(w_x))$ such that $\gamma_{w_x}(t_x)=x$. For each $w \in S_pM$ let $P_t^{w}\colon T_{p}M \to T_{\gamma_w(t)}M$ denote parallel translation along the geodesic $\gamma_w\colon\R\to M$. Define the vector field $V^{p}$ on $U_p$ by:
\begin{equation}\label{eq:V}
V^{p}(x):=\begin{cases}
\xi, & \text{ if } x=p \\
P^{w_x}_{t_x}(\xi), & \text{ if }x \in U_p\setminus \{p\}.
\end{cases}
\end{equation}
Define the line field $X^{p}$ on $U_p$ by:
\begin{equation}\label{eq:xp}
X^{p}:=\spn \{V^{p}\}.
\end{equation}

The following alternative description of $V^{p}$ on the subset $U_p \setminus \gamma_{\xi}(\R)$ will be useful. Let $v \in S_pM \cap \xi^{\perp}$ and consider the geodesic segment $c(s)$ through $\pm\xi$ and $v$ given by \eqref{eq:cs}. Define smooth functions $\bar{a}$ and $\bar{b}$ along the geodesic $c(s)$ by 
\begin{equation*}
\bar{a}\big(c(s)\big)=\sin(s) \quad \text{ and }\quad \bar{b}\big(c(s)\big)=\cos(s),
\end{equation*}
cf.\ \eqref{eq:defaabb}. Extend $\bar{a}$ and $\bar{b}$ to smooth functions on $S_pM$ invariant under rotations that fix $\{\pm \xi\}$. Consider the rotationally invariant orthonormal frame $\{e_1,e_2\}$ of $S_pM \setminus \{\pm \xi\}$, with $e_1$ tangent to the rank distribution $\mathcal R$, oriented by $e_1\big(c(s)\big)=\dot{c}(s)$. Let $\{E_0,E_1,E_2\}$ be the adapted frame on $U_p \setminus \gamma_{\xi}(\R)$ induced by $\{e_1,e_2\}$, as discussed in the beginning of this subsection. By construction, $\bar{a}$ and $\bar{b}$ satisfy \eqref{reduce1} and \eqref{reduce2} on $S_pM \setminus \{\pm \xi\}$ hence by Lemma~\ref{lemma:reduce} induce smooth radially constant functions $a$ and $b$ on $U_p \setminus \gamma_{\xi}(\R)$ satisfying \eqref{radial}, \eqref{rotational} and \eqref{trig}.

We claim that, on $U_p \setminus \gamma_{\xi}(\R)$, the vector field \eqref{eq:V} is given by:
\begin{equation}\label{eq:alt}
V^{p}=a\,E_0+b\,E_1.
\end{equation} 
By \eqref{rotational} and the rotational invariance of $\{e_1,e_2\}$, it suffices to verify the above along geodesic rays with initial velocity in the interior of the geodesic segment $c(s)$. For each $s \in \left(-\tfrac\pi2,\tfrac\pi2\right)$, it is easy to see that
\begin{equation}\label{eq:ver}
\xi=\bar{a}\big(c(s)\big)\,c(s)+\bar{b}\big(c(s)\big)\,\dot{c}(s).
\end{equation}
The parallel vector fields $E_0(t)$ and $E_1(t)$ along $\gamma_{c(s)}(t)$ are respectively $E_0(t)=P^{c(s)}_t\big(c(s)\big)$ and $E_1(t)=P^{c(s)}_t\big(e_1(c(s))\big)=P^{c(s)}_t\big(\dot{c}(s)\big)$. Therefore,
\begin{align*}
V^{p}\big(\gamma_{c(s)}(t)\big)&=P^{c(s)}_t(\xi)\\
&=P^{c(s)}_t\big(\bar{a}(c(s))c(s)+\bar{b}(c(s))\dot{c}(s)\big)\\
&=\bar{a}(c(s))P^{c(s)}_t\big(c(s)\big)+\bar{b}(c(s))P^{c(s)}_t\big(\dot{c}(s)\big)\\
&=\bar{a}(c(s))E_0(t)+\bar{b}(c(s))E_1(t)\\
&=a\big(\gamma_{c(s)}(t)\big)E_0(t)+b\big(\gamma_{c(s)}(t)\big)E_1(t),
\end{align*}
concluding the proof of \eqref{eq:alt}. Note this also follows from Proposition~\ref{prop:flats}, as parallel translation from $p\in\Sigma$ along $\Sigma$ is conjugate to parallel translation in $T_pM$ along $\sigma$, via $\dd(\exp_p)$. The line field $X^p$ is geometrically related to the above mentioned open book decomposition, as follows.

\begin{lemma}\label{lemma:foliate}
Let $p \in \mathcal{O}$ and $\sigma$ be a $2$-plane in $T_pM$ containing $L_p$.  The restriction of the line field $X^{p}$ given by \eqref{eq:xp} to the flat submanifold $\Sigma=\exp_p(\sigma)$ is tangent to the foliation of $\Sigma$ by lines parallel to $\gamma_{\xi}(\R)$, where $\xi \in L_p$ is a unit vector.
\end{lemma}

\begin{proof}
Without loss of generality, assume $v \in \sigma\cap\xi^\perp$ so that the geodesic segment $c(s)$ given by \eqref{eq:cs} lies in $S_pM \cap \sigma$. Let $V^\xi$ denote the parallel unit vector field on $\sigma$ determined by $\xi \in L_p$. Clearly $V^\xi$ is tangent to a foliation of $\sigma$ by straight lines parallel to $L_p$. This foliation is mapped under $\exp_p$ to a foliation of $\Sigma$ by lines parallel to $\gamma_{\xi}(\R)$. It remains to check that $\dd(\exp_p)_x\big(V^\xi\big)=V^{p}\big(\exp_p(x)\big)$, $x \in \sigma$. By continuity, it suffices to check this for $x \in \sigma \setminus L_p$. Assume $x=t\,c(s)$ with $s \in (-\tfrac\pi2,\tfrac\pi2)$ and $t\neq0$. Identify the orthonormal basis $\{c(s),\dot{c}(s)\}$ of $T_p \sigma$ with an orthonormal basis of $T_x\sigma$ and use \eqref{eq:ver} to deduce that
\begin{equation*}
V^\xi(x)=\bar{a}\big(c(s)\big)c(s)+\bar{b}\big(c(s)\big)\dot{c}(s).
\end{equation*}
Then, using \eqref{eq:E00E01}, \eqref{radial} and \eqref{eq:alt} respectively, we have:
\begin{align*}
\dd(\exp_p)_w\big(V^\xi\big)&=\bar{a}\big(c(s)\big)E_0(t)+\bar{b}\big(c(s)\big)E_1(t)\\
&=a\big(\gamma_{c(s)}(t)\big)E_0(t)+b\big(\gamma_{c(s)}(t)\big)E_1(t)\\
&=V^{p}\big(\exp_p(x)\big).\qedhere
\end{align*}
\end{proof}

\begin{figure}[htf]
\centering
\vspace{-0.4cm}
\includegraphics[scale=0.7]{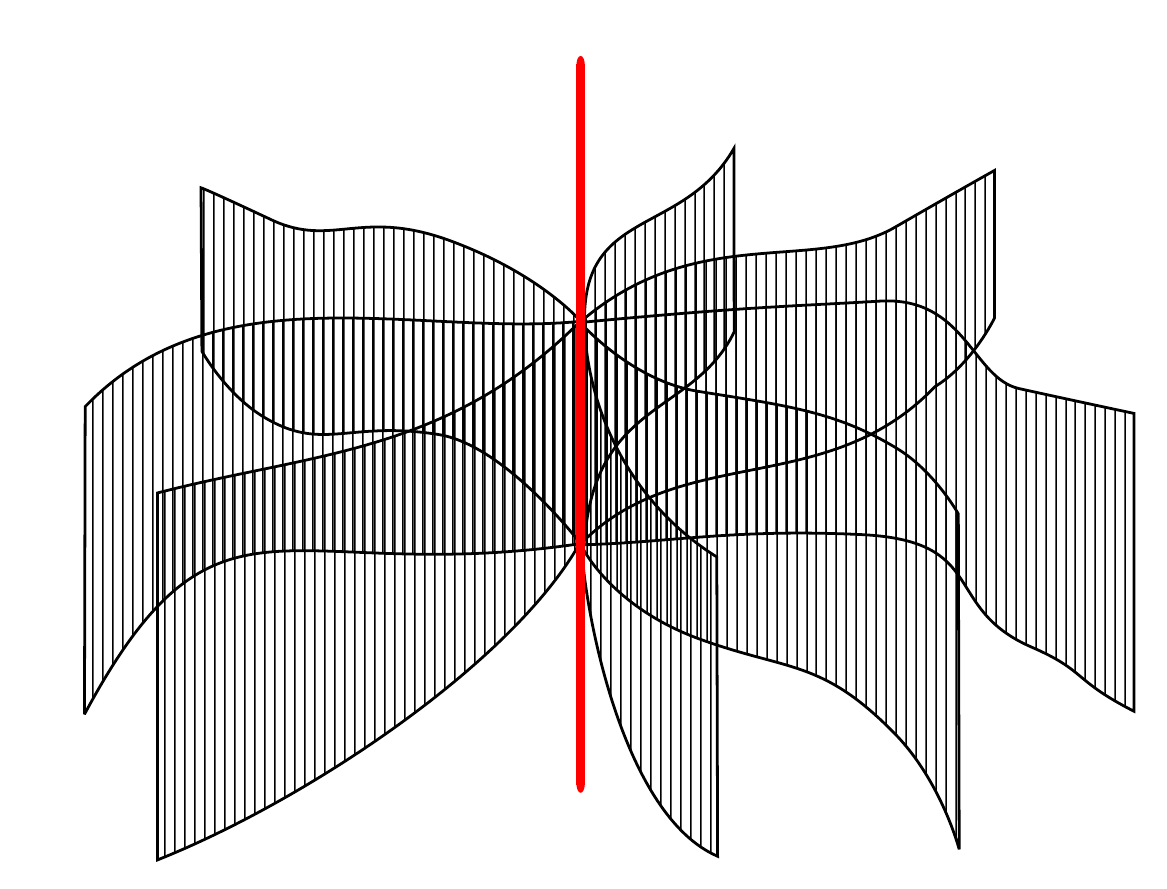}
\begin{pgfpicture}
\pgfputat{\pgfxy(-5.8,5.7)}{\pgfbox[center,center]{$\gamma_\xi(\R)=\exp_p(L_p)$}}
\pgfputat{\pgfxy(0,2.5)}{\pgfbox[center,center]{$\Sigma$}}
\end{pgfpicture}
\vspace{0.2cm}
\caption{Totally geodesic flat submanifolds $\Sigma$, foliated by lines parallel to the geodesic $\gamma_\xi(\R)$, pictured as the vertical red line.}
\label{figure:flats}
\vspace{-0.1cm}
\end{figure}

\begin{proposition}\label{prop:keyprop}
The line field $X^p$ given by \eqref{eq:xp} satisfies $X^{p}(x)=L_x$ for all $x \in U_p\cap \mathcal{O}$.
\end{proposition}

\begin{proof}
By Theorem \ref{thm:summary}, $V^{p}(\gamma_{\xi}(t))=P^{\xi}_t(\xi)=\dot{\gamma}_{\xi}(t) \in L_{\gamma_{\xi}(t)}$ for each $t \in \R$. Thus, $X^p(x)=L_x$ for all $x\in U_p \cap \gamma_{\xi}(\R)$. For $x \in \big(U_p \setminus \gamma_{\xi}(\R)\big)\cap\mathcal{O}$, there exists $v_x \in S_pM \setminus\{\pm \xi\}$ and $t_x \in (0,\mu(v_x))$ such that $\gamma_{v_x}(t_x)=x$. Consider the $2$-plane $\sigma=\spn\{\xi,v_x\}$ in $T_pM$. By Proposition~\ref{prop:flats}, $\Sigma=\exp_p(\sigma)$ is a totally geodesic flat immersed submanifold of $M$, and hence the $2$-plane $T_x \Sigma$ in $T_xM$ has zero sectional curvature. As the line $L_x$ in $T_x M$ lies in every $2$-plane of zero sectional curvature, we have that $L_x$ lies in $T_x \Sigma$. If $V^{p}(x) \notin L_x$, then the two lines $\gamma_{\xi}(\R)$ and $\exp_x(L_x)$ are not parallel in $\Sigma$, by Lemma~\ref{lemma:foliate}.  Consequently, they intersect transversally at some point $y\in\gamma_{\xi}(\R)$. By Lemma~\ref{lemma:oneline}, $y\in\mathcal I$, contradicting the fact that $\gamma_{\xi}(\R)\subset \mathcal{O}$ by Theorem~\ref{thm:summary}. Therefore, $V^{p}(x) \in L_x$ and hence $X^{p}(x)=L_x$.
\end{proof}

\begin{corollary}\label{corollary:perptotgeo}
The distribution $L^{\perp}$ defined on $\mathcal{O}$ is totally geodesic. 
\end{corollary}

\begin{proof}
Let $p \in \mathcal{O}$ and $v \in L_p^{\perp}$.  We must show that $\dot{\gamma}_v(t) \in L_{\gamma(t)}^{\perp}$ for all $t$ sufficiently small.  Let $\sigma$ be the $2$-plane in $T_pM$ containing $v$ and $L_p$.  By Proposition~\ref{prop:flats}, $\Sigma=\exp_p(\sigma)$ is a totally geodesic flat.  Therefore, the geodesic $\gamma_v(t)$ stays in $\Sigma$ and remains perpendicular to the foliation of $\Sigma$ by straight lines parallel to $\gamma_{\xi}(\R)=\exp_p(L_p)$.   By Lemma~\ref{lemma:foliate}, this foliation is tangent to the line field $X^p$, which agrees with the line field $L$ on $\mathcal{O}$ by Propostion~\ref{prop:keyprop}.
\end{proof}

\begin{corollary}\label{corollary:completeparallel}
The line field $L$ is parallel on each connected component of $\mathcal{O}$.
\end{corollary}

\begin{proof}
Immediate consequence of Proposition~\ref{prop:totgeoandparallel} and Corollary~\ref{corollary:perptotgeo}.
\end{proof}

\begin{proposition}\label{prop:keyprop2}
The line field $X^p$ given by \eqref{eq:xp} is parallel on $U_p=M \setminus \cut(p)$.
\end{proposition}

\begin{proof}
By Proposition ~\ref{prop:keyprop} and Corollary~\ref{corollary:completeparallel}, $X^{p}$ is parallel on $U_p\cap \mathcal{O}$, and it remains to show that $X^{p}$ is parallel on $U_p\cap\mathcal{I}$. As mentioned above, the functions $a$ and $b$ in \eqref{eq:alt} satisfy \eqref{radial}, \eqref{rotational} and \eqref{trig} on $U_p \setminus \gamma_{\xi}(\R)$. Since $X^p$ is parallel on $U_p\cap \mathcal{O}$, by Lemma~\ref{lemma:pde}, they satisfy \eqref{trans} on $\big(U_p \setminus \gamma_{\xi}(\R)\big)\cap \mathcal{O}$ and hence, by Lemma~\ref{isotoo}, $a$ and $b$ satisfy \eqref{radial}, \eqref{rotational}, \eqref{trig} and \eqref{trans} on all of $U_p \setminus \gamma_{\xi}(\R)$. The result now follows from Lemma~\ref{lemma:pde}. %, since $\mathcal{I} \subset   U_p \setminus \gamma_{\xi}(\R)$.
\end{proof}

\section{Proof of Theorem~\ref{thm:A}}\label{sec:proofA}
The strategy to prove Theorem \ref{thm:A} is to patch together the line fields constructed in Proposition~\ref{prop:keyprop} to construct a globally parallel line field on $M$. Then, the universal covering of $M$ splits a line as a consequence of de Rham decomposition theorem. For this, we will need the following:

\begin{lemma}\label{lemma:patch}
Let $p \in \mathcal{O}$ and $\sigma$ be a $2$-plane in $T_pM$ containing $L_p$. 
Consider $\Sigma=\exp_p(\sigma)$ and the line field $X^p$ given by \eqref{eq:xp}. Assume that $x \in \Sigma$ is an isotropic point and let $C \subset S_xM$ be the great circle $C=T_x \Sigma \cap S_xM$. Then:
\begin{enumerate}
\item\label{item:patch1} There are precisely two rank $3$ vectors in $C$, given by $C \cap X^{p}(x)$;
\item\label{item:patch2} For each $c \in C$, $T_cC$ is a subspace of the rank distribution $\mathcal R_c$. 
\end{enumerate}
\end{lemma}

\begin{proof}
By Lemma \ref{lemma:foliate}, the restriction of $X^{p}$ to $\Sigma$ is tangent to a foliation by lines parallel to $\gamma_{\xi}(\R)$. By Corollary~\ref{cor:nocircles}, $C$ does not consist entirely of rank $2$ vectors. If a rank $3$ vector $w \in C$ is not tangent to the line $X^{p}(x)$, then the two nonparallel lines $\gamma_{\xi}(\R)$ and $\exp_x(\R \cdot w)$ must intersect transversally at some point $y$.  This point $y$ is isotropic by Lemma~\ref{lemma:oneline}, contradicting the fact that $y \in \gamma_\xi(\R) \subset \mathcal{O}$. Thus, the rank $3$ vectors in $C$ are contained in $X^{p}(x)\cap C$. By definition, the rank of a unit vector $v$ coincides with the rank of $-v$, concluding the proof of \eqref{item:patch1}.

It remains to prove \eqref{item:patch2} for all rank $2$ vectors $c \in C$. As $x \in \Sigma$, there exists a unit vector $v \in S_p M \setminus \{\pm \xi\}$ and $t_0>0$ such that $\gamma_{v}(t_0)=x$. Let $\{E_0,E_1,E_2\}$ denote the adapted frame along $\gamma_v$ and consider the rank $2$ vector $E_0(t)=\dot{\gamma_v}(t_0) \in C$. Then $T_{E_0(t)}C$ is obtained by parallel translating in $T_x \Sigma$ the subspace spanned by $E_1(t)$. Thus, $E_0(t)$ satisfies \eqref{item:patch2}, and hence by Proposition~\ref{prop:totgeo-iso}, so do all rank $2$ vectors in $C$.
\end{proof}

\begin{proof}[Proof of Theorem~\ref{thm:A}]
If $\widetilde M$ splits isometrically as a product, then $M$ clearly has higher rank. Conversely, assume $M$ is a complete higher rank $3$-manifold. If $M$ consists entirely of isotropic points, then $M$ is flat and hence its universal covering is isometric to the Euclidean space $\R^3$.
Hence, assume $M$ has nonisotropic points. By the de Rham decomposition theorem, it suffices to construct a parallel line field $X$ on $M$. Let $B \subset \mathcal{O}$ denote a small metric ball in $M$. By Lemma~\ref{lemma:cover}, the open subsets $U_p=M \setminus \cut(p)$, $p \in B$, cover $M$. Propositions~\ref{prop:keyprop} and \ref{prop:keyprop2} guarantee that the line field $X^{p}$ on $U_p$ given by \eqref{eq:xp} is parallel and agrees with the line field $L$ at nonisotropic points.

We claim that $X^{p_1}=X^{p_2}$ on $U_{p_1}\cap U_{p_2}$ for any $p_1,p_2 \in B$. To prove the claim, let $x \in U_{p_1}\cap U_{p_2}$. If $x \in \mathcal{O}$, then $X^{p_1}(x)=L_x=X^{p_2}(x)$, by Proposition \ref{prop:keyprop}. Hence, assume $x \in \mathcal{I}$. As the geodesic $\exp_{p_i}(L_{p_i})$ consists entirely of nonisotropic points, there exist unique $v_x^i \in S_{p_i}M \setminus \{L_{p_i}\cap S_{p_i}M\}$ and $t_{x}^i \in (0,\mu(v_x^i))$ such that $\gamma_{v_x^i}(t_x^i)=x$, $i=1,2$. Consider the $2$-planes $\sigma_i:=\spn\{\xi_{p_i},v_x^i\}$ of $T_{p_i}M$, $i=1,2$. By Proposition~\ref{prop:flats}, $\Sigma_i=\exp_{p_i}(\sigma_i)$ is a totally geodesic flat immersed submanifold in $M$ that contains $p_i$ and $x$. Let $C_i$ denote the great circle $T_x \Sigma_i \cap S_xM$. If $C_1=C_2$ then Lemma \ref{lemma:patch} \eqref{item:patch1} implies that $X^{p_1}(x)=X^{p_2}(x)$. Otherwise, $C_1$ and $C_2$ must intersect transversally in a pair of antipodal vectors in $S_xM$. These antipodal vectors have rank $3$ by Lemma~\ref{lemma:patch} \eqref{item:patch2}. Lemma~\ref{lemma:patch} \eqref{item:patch1} then implies that $X^{p_1}(x)=X^{p_2}(x)$, concluding the proof of the claim.

By the above, there is a line field $X$ on $M$ whose restriction to $U_p$ agrees with $X^p$, for any $p\in B$. We conclude by showing that $X$ is parallel on $M$. Let $\tau\colon [0,1]\to M$ be a smooth curve and denote parallel translation along $\tau$ by
\begin{equation*}
P^{t_2}_{t_1}\colon T_{\tau(t_1)} M \to T_{\tau(t_2)}M
\end{equation*}
Set $X(t):=X(\tau(t))$ and $A:=\{t \in [0,1] : P^{t}_{0}(X(0))=X(t)\}$.  We must show that $1 \in A$. As $P^{0}_{0}=\Id$, $0 \in A$. Continuity of parallel translation implies that $A$ is closed. To see that $A$ is also open, pick $t_0 \in A$. By the covering property, there exists $p \in B$ and $\varepsilon>0$ such that $\tau\big((t_0-\varepsilon,t_0+\varepsilon)\big)\subset U_p$. Let $s \in (t_0-\varepsilon,t_0+\varepsilon)$. Using that $t_0 \in A$, that $X^p$ is parallel and that $X$ restricts to $X^p$ on $U_p$, we have:
\begin{equation*}
X(s)=X^{p}\big(\tau(s)\big)=P^{s}_{t_0}\big(X^{p}(\tau(t_0))\big)= P^{s}_{t_0}\big(X(t_0)\big)=P^{s}_{t_0}\big(P_{0}^{t_0}(X(0))\big)=P_{0}^{s}(X(0)),
\end{equation*}
hence $(t_0-\varepsilon,t_0+\varepsilon)\subset A$. Therefore $A=[0,1]$, concluding the proof.
\end{proof}

\section{\texorpdfstring{Gluing constructions of manifolds with $\cvc(0)$}{Gluing constructions of manifolds with cvc(0)}}\label{sec:graphmflds}

In this section, we describe metrics with $\cvc(0)$ and pointwise signed sectional curvatures on both closed and open $3$-manifolds, via gluing constructions. While all of these examples have a local product decomposition, many have irreducible universal covering.

\subsection{Closed examples}
Graph manifolds were first considered by Waldhausen~\cite{wald} in the 1960s and have since been used to construct important examples in various contexts, most notably by Gromov~\cite{gromov} and Cheeger and Gromov~\cite{cheeger-gromov}. These are $3$-manifolds obtained by gluing circle bundles over surfaces with boundary. %It follows from the solution of the Geometrization Conjecture that graph manifolds are precisely the $3$-manifolds that have zero simplicial volume, since their geometric decomposition admits no hyperbolic parts.

\begin{definition}\label{def:graphmfld}
Consider finitely many surfaces $\Sigma_i^2$ whose boundary is a disjoint union of circles $\partial \Sigma_i^2=\bigcup_{j=1}^{N(i)} S^1_{i,j}$. The boundary components of the product manifold $M_i:=\Sigma_i^2\times S^1_i$ are tori $S^1_{i,j}\times S^1_i$, $1\leq j\leq N(i)$. A \emph{graph manifold} $M=\bigcup_i M_i$ is a $3$-manifold obtained by gluing the $M_i$ together along pairs of boundary tori, which are identified via an orientation reversing diffeomorphism.
\end{definition}

Note that by multiplying any element $A\in\SL(2,\Z)$ on the left with the reflection
\begin{equation*}
\tau=\begin{pmatrix}
-1 & 0\\
0 & 1
\end{pmatrix},
\end{equation*}
we obtain an orientation reversing diffeomorphism $\tau A\colon\R^2\to\R^2$ that leaves invariant the integer lattice $\Z^2$. Thus, $\tau A$ descends to an orientation reversing diffeomorphism $\varphi_A$ of the torus $T^2=\R^2/\Z^2$. Conversely, any orientation reversing diffeomorphism $\varphi\colon T^2\to T^2$ is isotopic to $\varphi_A$ for some $A\in\SL(2,\Z)$. %Therefore, every gluing map used to define a graph manifold can be thought of as an element of $\SL(2,\Z)$.

The information necessary to define a graph manifold can be conveniently organized in the form of a graph, hence the name. Each $M_i$ corresponds to a vertex of this graph, labeled with the surface $\Sigma_i$. There are $N(i)$ edges that issue from this vertex, each decorated with an element of $\SL(2,\Z)$ that encodes the respective gluing map, as explained above.

\begin{convention}
Every torus boundary component of $M_i=\Sigma_i^2\times S^1_i$ is written as $S^1_{i,j}\times S^1_i$, meaning that the first factor $S^1_{i,j}$ is a part of $\partial\Sigma_i^2$ and the second factor $S^1_i$ is the circle fiber of the trivial bundle $S^1\to M_i\to\Sigma_i^2$.
\end{convention}

Consider the cyclic subgroup of order $4$ of $\SL(2,\Z)$ generated by the $90^\circ$ rotation
\begin{equation*}
R_{\frac\pi2}:=\begin{pmatrix}
0 & -1\\
1 & 0
\end{pmatrix}, \quad \left\langle R_{\frac\pi2} \right\rangle\cong\Z_4\subset\SL(2,\Z).
\end{equation*}
We now analyze gluing maps given by orientation reversing diffeomorphisms of $T^2=\R^2/\Z^2$ which are induced by the orientation reversing diffeomorphisms
\begin{equation}\label{eq:AandB}
A:=\tau R_{\frac\pi2}^2=-\tau  R_{\frac\pi2}^4=\begin{pmatrix}
1 & 0\\
0 & -1
\end{pmatrix}\quad \mbox{ and }\quad B:=\tau R_{\frac\pi2}=-\tau R_{\frac\pi2}^3=\begin{pmatrix}
0 & 1\\
1 & 0
\end{pmatrix}.
\end{equation}
With the above convention, if the gluing map $\varphi\colon S^1_{i_1,j_1}\times S^1_{i_1}\to S^1_{i_2,j_2}\times S^1_{i_2}$ between torus boundary components of $M_{i_1}$ and $M_{i_2}$ is $\varphi_A$, then the circle components of $\partial \Sigma_{i_1}$ and $\partial\Sigma_{i_2}$ are identified with one another, and so are the circle fibers $S^1_{i_1}$ and $S^1_{i_2}$.  However, with the gluing map $\varphi_B$, the circle component of $\partial\Sigma_{i_1}$ is identified with the circle fiber $S^1_{i_2}$ and the circle component of $\partial\Sigma_{i_2}$ is identified with the circle fiber $S^1_{i_1}$. In other words, $\varphi_A$ is a \emph{trivial} gluing map that preserves vertical and horizontal directions of $S^1\to M_i\to\Sigma_i^2$, while $\varphi_B$ is a gluing map that interchanges vertical and horizontal directions.

\begin{proposition}\label{prop:graphcvc0}
If $M$ is a graph manifold all of whose gluing maps are the diffeomorphisms $\varphi_A$ or $\varphi_B$ induced by \eqref{eq:AandB}, then $M$ admits a metric with $\cvc(0)$ and pointwise signed sectional curvatures.
\end{proposition}

\begin{proof}
In the notation of Definition~\ref{def:graphmfld}, endow the surfaces $\Sigma_i^2$ with any smooth metric that restricts to a product metric on a collar neighborhood of each component of the boundary $\partial \Sigma_i^2=\bigcup_{j=1}^{N(i)} S^1_{i,j}$. Without loss of generality, assume this is done in such way that each $S^1_{i,j}$ is a circle of unit length. Then, consider $M_i=\Sigma^2_i\times S^1_i$ endowed with the product metric. The gluing maps $\varphi_A$ and $\varphi_B$ are isometries of the square torus, and hence the above metrics on $M_i$ can be glued together. The resulting metric on $M$ clearly has $\cvc(0)$, since every point has a neighborhood isometric to a product.
\end{proof}

\begin{figure}[htf]
\centering
\vspace{-0.5cm}
\includegraphics[scale=0.6]{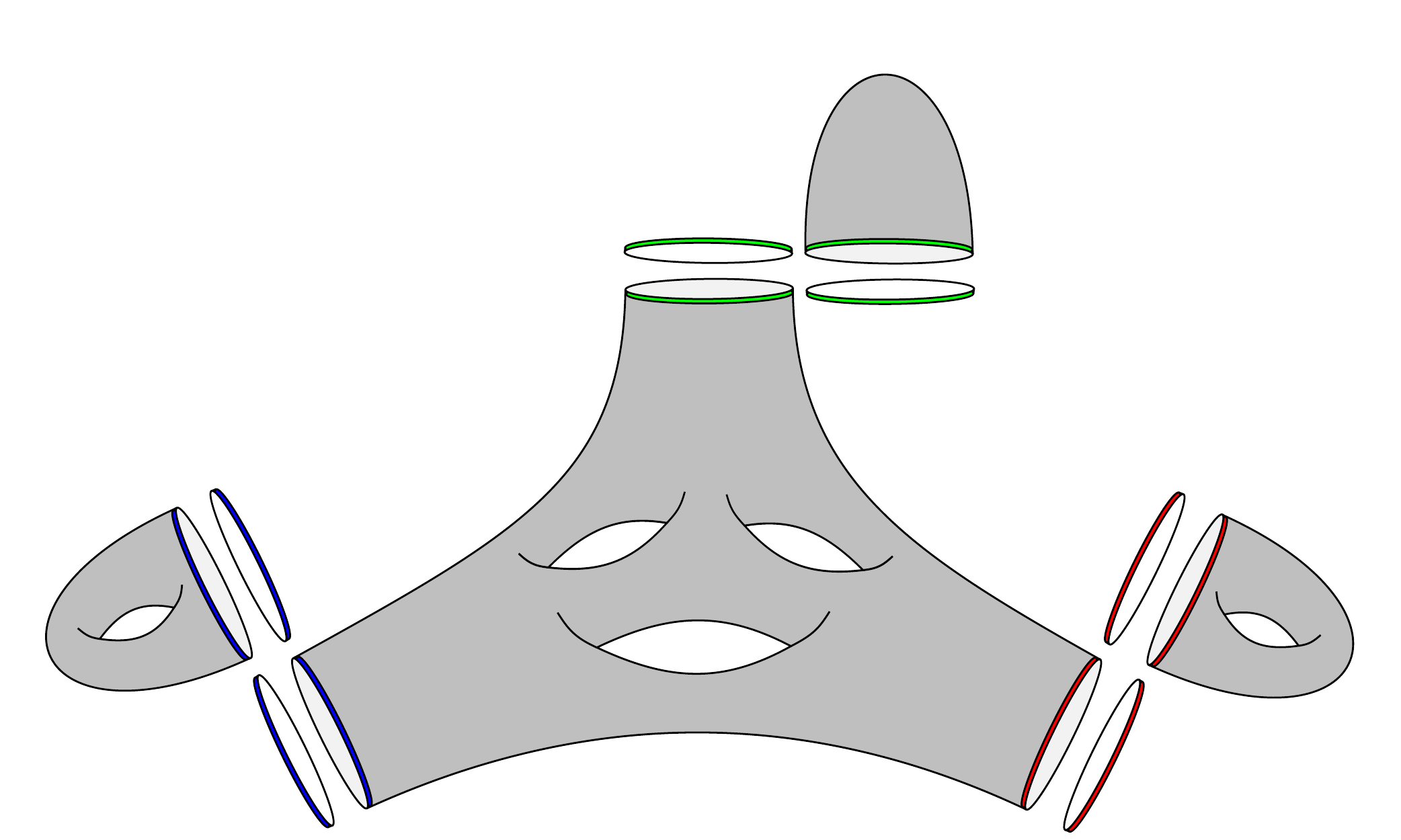} %uncomment to create a happy graph manifold
\begin{pgfpicture}
\pgfputat{\pgfxy(-6,2.8)}{\pgfbox[center,center]{$\Sigma_1$}}
\pgfputat{\pgfxy(-4.25,3.9)}{\pgfbox[center,center]{$S^1_{2}$}}
\pgfputat{\pgfxy(-4.85,3.75)}{\pgfbox[center,center]{$S^1_{1,1}$}}
\pgfputat{\pgfxy(-3.25,0.2)}{\pgfbox[center,center]{$S^1_{1}$}}
\pgfputat{\pgfxy(-2.65,0.4)}{\pgfbox[center,center]{$S^1_{2,1}$}}
\pgfputat{\pgfxy(0.1,0.9)}{\pgfbox[center,center]{$\Sigma_2$}}
\pgfputat{\pgfxy(2.6,7.2)}{\pgfbox[center,center]{$\Sigma_3$}}
\pgfputat{\pgfxy(3,5.8)}{\pgfbox[center,center]{$S^1_{3,1}$}}
\pgfputat{\pgfxy(2.9,5.3)}{\pgfbox[center,center]{$S^1_{2}$}}
\pgfputat{\pgfxy(-1.1,5.9)}{\pgfbox[center,center]{$S^1_{3}$}}
\pgfputat{\pgfxy(-1.1,5.4)}{\pgfbox[center,center]{$S^1_{2,2}$}}
\pgfputat{\pgfxy(6.1,2.8)}{\pgfbox[center,center]{$\Sigma_4$}}
\pgfputat{\pgfxy(4.5,3.8)}{\pgfbox[center,center]{$S^1_{2}$}}
\pgfputat{\pgfxy(5.1,3.6)}{\pgfbox[center,center]{$S^1_{4,1}$}}
\pgfputat{\pgfxy(3.5,0.15)}{\pgfbox[center,center]{$S^1_{4}$}}
\pgfputat{\pgfxy(2.9,0.3)}{\pgfbox[center,center]{$S^1_{2,3}$}}
\end{pgfpicture}
\vspace{-0.1cm}
\caption{A graph manifold $M=\bigcup_{i=1}^4 M_i$ with $\cvc(0)$. The corresponding minimal graph decomposition has one central vertex labeled $\Sigma_2$ with three edges, each terminating on a vertex labeled $\Sigma_1$, $\Sigma_3$ or $\Sigma_4$.}
\label{figure:graphmanifold}
\end{figure}

%\begin{remark}
%Recall that all graph manifolds in the sense of Definition~\ref{def:graphmfld}, in particular the examples discussed in Proposition~\ref{prop:graphcvc0}, admit an $F$-structure of higher rank and hence collapse with bounded curvature \cite{cheeger-gromov}.
%\end{remark}

In the above construction, suppose that the pieces $M_{i_1}$ and $M_{i_2}$ of a graph manifold $M$ only share one torus boundary component, which is identified using the \emph{trivial} gluing map $\varphi_A$. Then $M_{i_1}\cup M_{i_2}$ is isometric to $(\Sigma_{i_1}\cup\Sigma_{i_2})\times S^1$, and hence we could have started with a smaller decomposition of $M$ as a graph manifold, in which $M_{i_1}$ and $M_{i_2}$ are already glued together. Thus, for the purpose of constructing $\cvc(0)$ metrics via Proposition~\ref{prop:graphcvc0}, we may consider a \emph{minimal graph decomposition}, in which all edges corresponding to the gluing map $\varphi_A$ are collapsed, and all remaining edges correspond to the gluing map $\varphi_B$. See Figures~\ref{figure:graphmanifold} and \ref{figure:s3} for examples of graph manifolds with $\cvc(0)$ presented by their minimal graph decomposition.

Note that if the minimal graph decomposition of a graph manifold $M$ with $\cvc(0)$ consists of only one vertex, then the universal covering of $M$ splits isometrically as a product. However, in general, these graph manifolds can have irreducible universal covering. Although such manifolds admit a metric with $\cvc(0)$ and pointwise signed sectional curvatures, which is a \emph{pointwise} notion of higher rank, it follows from Theorem~\ref{thm:A} that they do not admit metrics of higher rank. In particular, we have:

\begin{corollary}\label{cor:spherecvc0}
The sphere $S^3$ and all lens spaces $L(p;q)\cong S^3/\Z_p$ admit metrics with $\cvc(0)$ and $\sec\geq0$.
\end{corollary}

\begin{proof}
The genus $1$ Heegaard decomposition of $S^3$ provides a minimal graph decomposition, consisting of two vertices connected by one edge. More precisely, in this decomposition $S^3=M_1\cup M_2$, where $\Sigma_i^2=D^2_i$ are disks and the gluing map is $\varphi_B\colon \partial D^2_1\times S^1_1\to \partial D^2_2\times S^1_2$, see Figure~\ref{figure:s3}. Choosing metrics on $D^2_i$ that have $\sec\geq0$, we obtain (as in the proof of Proposition~\ref{prop:graphcvc0}) a metric on $S^3$ with $\cvc(0)$ and $\sec\geq0$.
\begin{figure}[htf]
\centering
\vspace{0.2cm}
\includegraphics[scale=0.4]{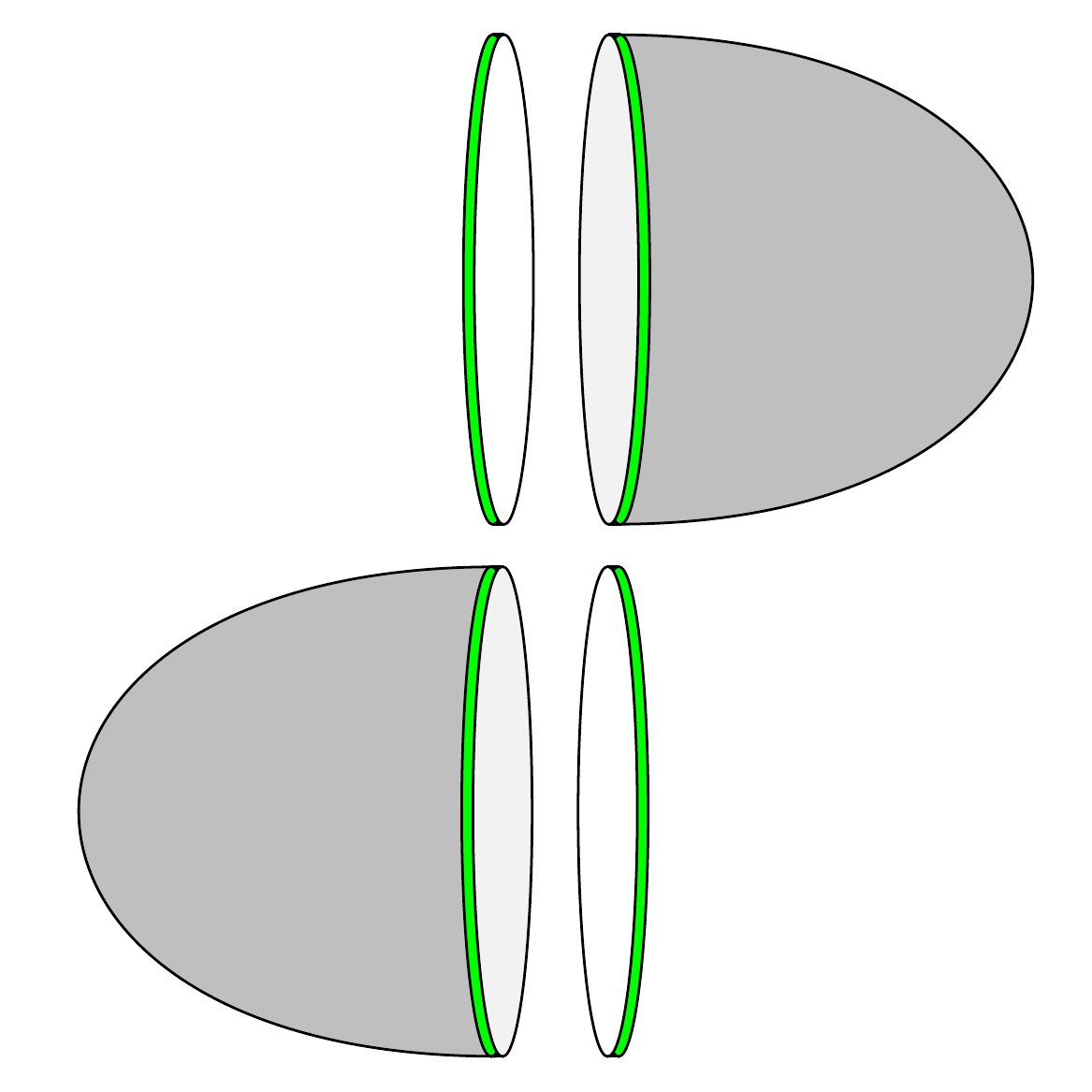}
\begin{pgfpicture}
\pgfputat{\pgfxy(-2.8,4.9)}{\pgfbox[center,center]{$S^1_1$}}
\pgfputat{\pgfxy(-2,4.87)}{\pgfbox[center,center]{$S^1_{2,1}$}}
\pgfputat{\pgfxy(-2.8,-0.153)}{\pgfbox[center,center]{$S^1_{1,1}$}}
\pgfputat{\pgfxy(-2.1,-0.15)}{\pgfbox[center,center]{$S^1_2$}}
\pgfputat{\pgfxy(-5,1.2)}{\pgfbox[center,center]{$D^2_1$}}
\pgfputat{\pgfxy(0.1,3.55)}{\pgfbox[center,center]{$D^2_2$}}
\end{pgfpicture}
\vspace{0.2cm}
\caption{The sphere $S^3$ seen as a graph manifold.}
\label{figure:s3}
\end{figure}
Furthermore, if the metrics on $D^2_i$ are rotationally symmetric, then the resulting metric on $S^3$ is invariant under the $T^2$-action $(e^{i\theta},e^{i\phi})\cdot (z,w)=(e^{i\theta}z,e^{i\phi}w)$, where $(e^{i\theta},e^{i\phi})\in T^2$ and $S^3=\{(z,w)\in\C^2:|z|^2+|w|^2=1\}$. In particular, this metric is also invariant under the subaction of the cyclic subgroup of order $p$ of $T^2$ generated by $(e^{2\pi i/p},e^{2\pi i q/p})$, where $\gcd(p,q)=1$. Thus, it descends to a metric with $\cvc(0)$ and $\sec\geq0$ on the quotient, which is the lens space $L(p;q)$.
\end{proof}

\begin{remark}
The lens space $L(p;q)$ is itself a graph manifold, obtained from gluing together two solid tori $M_1=D^2_1\times S^1_1$ and $M_2=D^2_2\times S^1_2$ using the orientation reversing diffeomorphism $\varphi\colon T^2\to T^2$ induced by
\begin{equation*}
\begin{pmatrix}
-q & r\\
p & s
\end{pmatrix}
\end{equation*}
where $r,s\in\Z$ are such that $pr+qs=1$. Since the above is not an isometry of the square torus $T^2=\R^2/\Z^2$, we cannot directly apply the construction of Proposition~\ref{prop:graphcvc0}. However, a direct construction of $\cvc(0)$ metrics on $L(p;q)$ in this framework is still possible, using \emph{twisted cylinders} and more general gluing maps~\cite{florit-ziller}.
\end{remark}

%\begin{remark}\label{rem:complicated}
%Much more complicated metrics with $\cvc(0)$ and $\sec\geq0$ can be constructed on $S^3$. In fact, the metrics on $D^2_i$ with $\sec\geq0$ can have arbitrarily many regions with zero curvature. These are obtained flattening the above metric on these prescribed regions, which may have complicated topology. The resulting metric on $S^3$ has $\cvc(0)$ and $\sec\geq0$, but $\mathcal I$ is not only the tubular neighborhood of the middle torus and $\mathcal O$ may have arbitrarily many components. Furthermore, if this is done with rotationally symmetric metrics on $D^2_i$, e.g., by flattening an interior annulus, then the resulting metric also descends to $L(p;q)$.
%\end{remark}

\begin{corollary}\label{cor:s2s1}
The product manifold $S^2\times S^1$ admits metrics with $\cvc(0)$ and pointwise signed sectional curvatures that do not have higher rank.
\end{corollary}

\begin{proof}
Consider the graph manifold whose minimal graph decomposition is obtained from the minimal graph decomposition of $S^3$ mentioned above by adding one vertex along the edge between the two original vertices, see Figure~\ref{figure:s2s1}.
\begin{figure}[htf]
\centering
\vspace{0.1cm}
\includegraphics[scale=0.4]{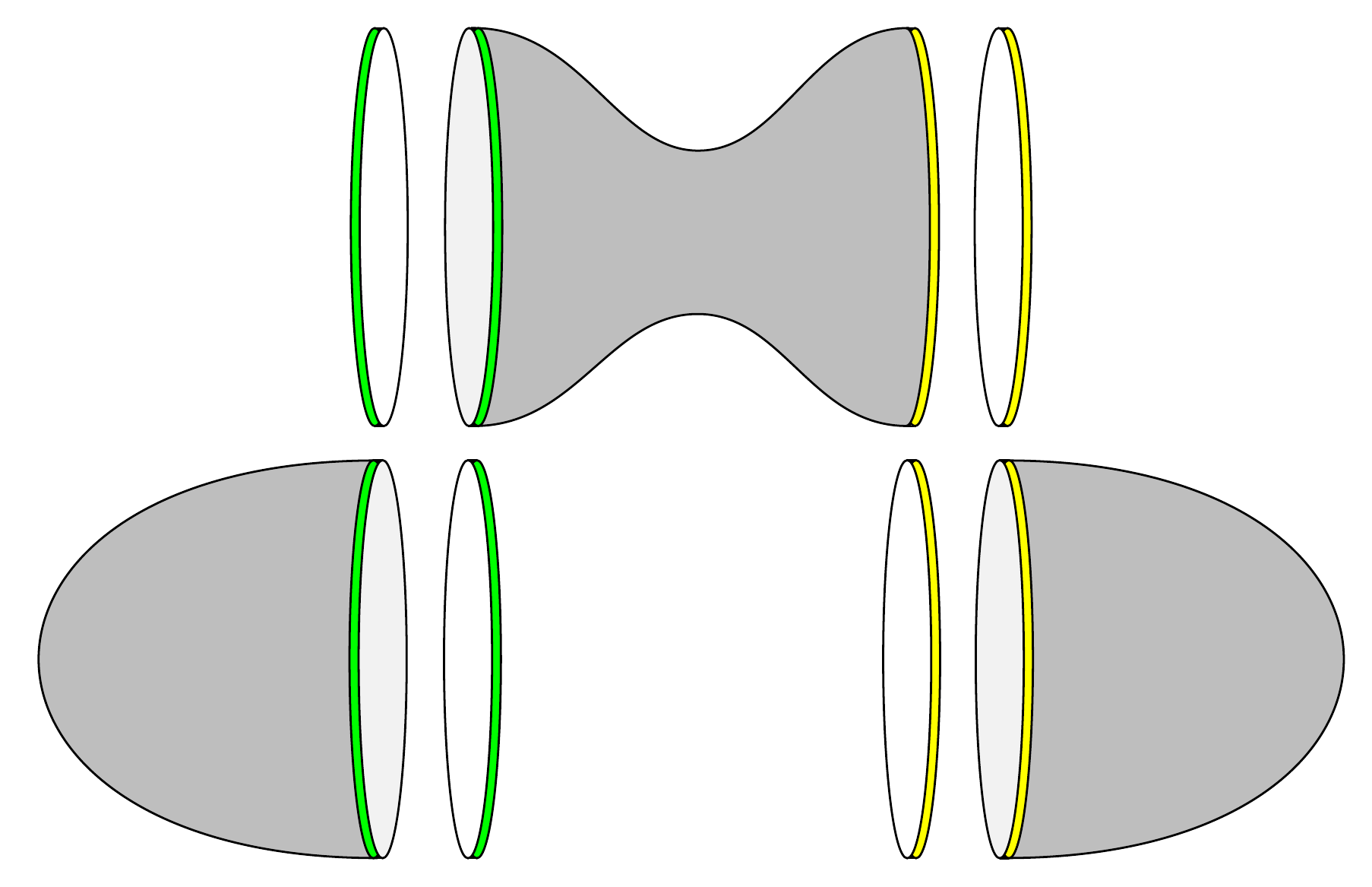}
\begin{pgfpicture}
\pgfputat{\pgfxy(-2.6,4.9)}{\pgfbox[center,center]{$S^1_{2,2}$}}
\pgfputat{\pgfxy(-1.9,4.87)}{\pgfbox[center,center]{$S^1_{3,1}$}}
\pgfputat{\pgfxy(-2.6,-0.153)}{\pgfbox[center,center]{$S^1_{2}$}}
\pgfputat{\pgfxy(-1.9,-0.153)}{\pgfbox[center,center]{$S^1_3$}}
\pgfputat{\pgfxy(-7.7,1.2)}{\pgfbox[center,center]{$\Sigma_1$}}
\pgfputat{\pgfxy(-3.75,4.5)}{\pgfbox[center,center]{$\Sigma_2$}}
\pgfputat{\pgfxy(0.2,1.2)}{\pgfbox[center,center]{$\Sigma_3$}}
\pgfputat{\pgfxy(-5.5,4.9)}{\pgfbox[center,center]{$S^1_1$}}
\pgfputat{\pgfxy(-4.8,4.87)}{\pgfbox[center,center]{$S^1_{2,1}$}}
\pgfputat{\pgfxy(-5.5,-0.153)}{\pgfbox[center,center]{$S^1_{1,1}$}}
\pgfputat{\pgfxy(-4.8,-0.153)}{\pgfbox[center,center]{$S^1_2$}}
\end{pgfpicture}
\vspace{0.1cm}
\caption{A nontrivial graph manifold decomposition of $S^2\times S^1$.}
\label{figure:s2s1}
\end{figure}
This $3$-manifold is clearly diffeomorphic to $S^2\times S^1$, by collapsing the $\Sigma_2$ cylinder portion. Endowing each $\Sigma_i$ of this minimal graph decomposition with non-flat metrics, we obtain a metric on $S^2\times S^1$ which has $\cvc(0)$, just as in Proposition~\ref{prop:graphcvc0}. Moreover, these metrics do not have higher rank. For instance, take $\gamma$ a geodesic that joins points $p_1\in \Sigma_1\times S^1_1$ and $p_2\in \Sigma_2\times S^1_2$ that lie in non-flat regions. Then, the line field $L$ along $\gamma$ cannot be parallel, since near $p_1$ it is tangent to $S^1_1$ and near $p_2$ it is tangent to $S^1_2$.
\end{proof}

%Notice that the above construction can be made more complicated by iterating the process, i.e., including more alternating non-flat cylinders similar to $\Sigma_2$ in Figure~\ref{figure:s2s1}. As a result, the metrics on $S^2\times S^1$ with $\cvc(0)$ that do not have higher rank can have arbitrarily complicated isotropic and nonisotropic sets.
\begin{remark}\label{rem:complicated2}
There are no such metrics with $\sec\geq0$, or even $\Ric\geq0$. This is a consequence of the Cheeger-Gromoll splitting theorem, as the universal covering must contain a line. Alternatively, note that a metric with $\sec\geq0$ on $S^1\times [-1,1]$ with geodesic boundary components is flat by Gauss-Bonnet. Hence, the local product structures in neighborhoods of the gluing loci extend across the cylinder.
\end{remark}

\subsection{Open examples}
The above gluing constructions can also be applied on open manifolds. For this, we need the following auxiliary result.

\begin{lemma}\label{lemma:strangeR2}
The upper half plane $\R^2_+$ admits smooth Riemannian metrics $\h$ with \emph{quasi-positive curvature}; i.e., $\sec\geq0$ and $\sec>0$ at a point, such that $\partial\R^2_+$ is totally geodesic and the metric is product on a collar neighborhood of $\partial\R^2$.
\end{lemma}
%\begin{lemma}
%The upper half plane $\R^2_+$ admits smooth Riemannian metrics $\h_*$ that restrict to a product (flat) metric on a collar neighborhood of the boundary $\partial\R^2$, such that $\partial\R^2_+$ is totally geodesic and $(\R^2_+,\h_*)$ has \emph{quasi-positive curvature}; i.e., $\sec\geq0$ and $\sec>0$ at a point.
%\end{lemma}

\begin{proof}
The desired metrics can be obtained by smoothing a standard construction in Alexandrov geometry. Consider the \emph{double} of the first quadrant $Q=\{(x,z):x\geq0,z\geq0\}$; i.e., two copies of $Q$ glued along the boundary. 
We can smooth this object along the gluing interface, in such way that the resulting surface $S$ is the boundary of a smooth convex region in $\R^3$, symmetric with respect to reflection on the $(x,z)$-plane. This surface has positive curvature near the point $o$ where the two origins $(0,0)\in Q$ were identified, and is flat on the complement of a compact set containing $o$. In particular, it contains two copies of $Q$ that lie in planes parallel to the $(x,z)$-plane, see Figure~\ref{figure:convex}.
\begin{figure}[htf]
\centering
\includegraphics[scale=0.35]{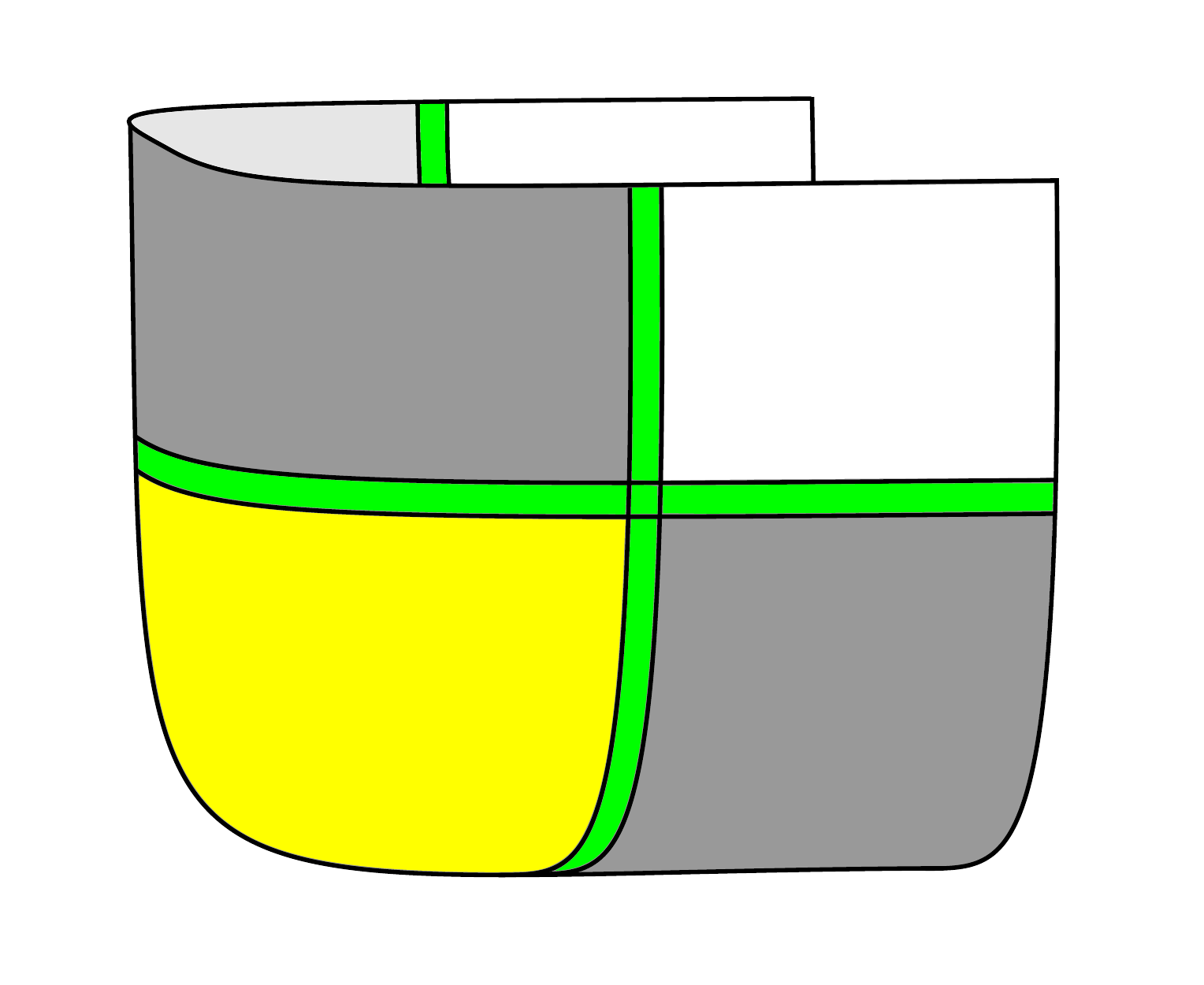}
\begin{pgfpicture}
\pgfputat{\pgfxy(-4.7,0.45)}{\pgfbox[center,center]{$o$}}
\pgfputat{\pgfxy(-1.6,2.9)}{\pgfbox[center,center]{$Q$}}
%\pgfputat{\pgfxy(-0.4,3.5)}{\pgfbox[center,center]{$\Pi_1$}}
%\pgfputat{\pgfxy(-1.5,4)}{\pgfbox[center,center]{$\Pi_2$}}
\end{pgfpicture}
\caption{The surface $S$ obtained by smoothing the double of $Q$.}
\label{figure:convex}
\end{figure}
The curve along which $S$ intersects a plane parallel to the $(x,y)$-plane (or the $(y,z)$-plane) is a geodesic, provided it is sufficiently far from $o$. By cutting $S$ along one such geodesic, we obtain a surface with boundary in $\R^3$, diffeomorphic to the upper half plane. The induced metric on $\R^2_+$ that makes this embedding isometric is the desired metric $\h$.
\end{proof}

\begin{proposition}\label{prop:r3}
There exist metrics on $\R^3$ with $\cvc(0)$ and $\sec\geq0$ that do not have higher rank.
\end{proposition}

\begin{proof}
We use a gluing procedure analogous to that of Proposition~\ref{prop:graphcvc0}. Consider the decomposition $\R^3=\R^3_-\cup\R^3_+$, where $\R^3_-=\{(x,y,z):z\leq0\}$ and $\R^3_+=\{(x,y,z):z\geq0\}$. 
Define a product metric $(\R^3_+,\g)=(\R^2_+,\h)\times\R$, where $\h$ is given by Lemma~\ref{lemma:strangeR2}. Note that $\g$ restricts to a product metric on a collar neighborhood of the boundary $\partial \R^3_+=\{(x,y,0)\}$, such that $\partial\R^3_+$ is totally geodesic and isometric to flat Euclidean space. The desired metric on $\R^3$ is then obtained by endowing both half-spaces $\R^3_\pm$ with the metric $\g$ and gluing them together via the identification $B\colon\partial\R^3_-\to\partial\R^3_+$, $B(x,y)=(y,x)$, cf.\ \eqref{eq:AandB}. This metric clearly has $\cvc(0)$, and does not have higher rank by an argument totally analogous to that in Corollary~\ref{cor:s2s1}, considering a geodesic $\gamma$ that joins nonisotropic points $p_\pm\in\R^3_\pm$.
\end{proof}

Together with Corollary~\ref{cor:spherecvc0}, the above completes the proof of Theorem~\ref{thm:B}.

\begin{remark}\label{rem:complicated3}
A number of modifications in the construction of $\h_*$ lead to other interesting examples of metrics on $\R^3$ with $\cvc(0)$ and without higher rank, via the process described in Proposition~\ref{prop:r3}. For instance, $(\R^2_+,\h)$ can be constructed with an unbounded region of positive curvature. This is achieved by smoothing the double of the convex set $C_f:=\{(x,y):y\geq f(x)\}$, where $f\colon\R\to\R$ is a smooth function with $f'(x)<0$ and $f''(x)<0$ for $x<0$ and $f\equiv 0$ for $x\geq0$, and then cutting along any geodesic curve corresponding to $x=\ell>0$.

Let $L$ be the closure of the complement of the first quadrant $Q\subset\R^2$. By smoothing the double of $L$, we obtain metrics on $\h$ on $\R^2_+$ with \emph{quasi-negative curvature}; i.e., $\sec\leq0$ and $\sec<0$ at a point. The resulting metrics on $\R^3$ have $\cvc(0)$ and $\sec\leq0$, but do not have higher rank. Similar examples with unbounded regions of negative curvature can be constructed using the closure of the complement of $C_f$ as above. Finally, examples with mixed (but pointwise signed) sectional curvatures and infinitely many nonisotropic components can be constructed using the closure of $\{(x,y):y\leq\lfloor x\rfloor\}$, where $\lfloor x\rfloor\in\Z$ denotes the largest integer smaller than or equal to $x$.
\end{remark}

\section{Real-analyticity and Theorem~\ref{thm:C}}\label{sec:proofC}

The above constructions of metrics with $\cvc(0)$ do not produce \emph{real-analytic} metrics. We now prove that these constructions cannot be made real-analytic among manifolds with finite volume and pointwise signed sectional curvatures.

\begin{proof}[Proof of Theorem~\ref{thm:C}]
Let $(\widetilde M,\widetilde\g)$ be the Riemannian universal covering of $(M,\g)$. The result is trivially true if $(\widetilde M,\widetilde\g)$ is flat. Else, there exists a small metric ball $B$ in $\widetilde M$ all of whose points are nonisotropic. By Theorem~\ref{thm:summary}, there exists a parallel vector field $V$ on $B$ that spans the rank $3$ line field $L$. Since $(\widetilde M,\widetilde \g)$ is real-analytic and simply-connected, a classical result of Nomizu~\cite{nomizu} implies that the local Killing field $V$ admits a (unique) extension to a global Killing field on $\widetilde M$, which we also denote by $V$. To show that $V$ is parallel, choose $p\in B$, $q\in\widetilde M$, $w\in T_q\widetilde M$, and a real-analytic curve $\gamma\colon [0,1]\to\widetilde M$ with $\gamma(0)=p$ and $\gamma(1)=q$. 
Let $W(t)$ be a real-analytic vector field along $\gamma(t)$ with $W(1)=w$, for instance the parallel transport of $w$ along $\gamma(t)$.
The vector field $\nabla_{W(t)}V$ along $\gamma(t)$ vanishes for $0<t<\varepsilon$ such that $\gamma(t)\in B$, and hence vanishes identically by real-analyticity. In particular, $\nabla_w V=0$. Therefore, the Killing field $V$ is globally parallel and hence $\widetilde M$ splits isometrically as a product, by the de Rham decomposition theorem.
\end{proof}

\begin{remark}
Note that the above result also holds if the real-analytic Riemannian manifold $M$ has infinite volume, provided $M$ has a nonisotropic component with finite volume.
\end{remark}


\begin{thebibliography}{99}
\bibitem{ba}{{\sc W.\ Ballmann}, \emph{Nonpositively curved manifolds of higher rank}, Ann.\ of Math.\ (2) 122 (1985), no. 3, 597-609.}
\bibitem{busp}{{ \sc K.\ Burns and R.\ Spatzier}, \emph{Manifolds of nonpositive curvature and their buildings}, Inst.\ Hautes Etudes Sci.\ Publ.\ Math. 65 (1987), 35-59.}
\bibitem{cheeger-gromov}{{\sc J. Cheeger \and M. Gromov}, \emph{Collapsing Riemannian manifolds while keeping their curvature bounded.} I. J. Differential Geom. 23 (1986), no. 3, 309-346.}
\bibitem{chi}{{\sc Q.-S. Chi}, \emph{A curvature characterization of certain locally rank-one symmetric spaces}, J. Differential Geom. 28 (1988), no.2, 187-202.}
\bibitem{co}{{\sc C. Connell}, \emph{A characterization of hyperbolic rank one negatively curved homogeneous spaces}, Geom. Dedicata 128 (2002), 221-246.}
\bibitem{con}{{\sc D. Constantine}, \emph{$2$-frame flow dynamics and hyperbolic rank-rigidity in nonpositive curvature}, J. Mod. Dyn. 2 (2008), no. 4, 719-740.}
\bibitem{ebhe}{{\sc P. Eberlein \and J. Heber}, \emph{A differential geometric characterization of symmetric spaces of higher rank.} Publ. IHES 71 (1990), 33-44.}
\bibitem{esol}{{\sc J. Eschenburg \and C. Olmos}, \emph{Rank and symmetry of Riemannian manifolds.} Comment. Math. Helvetici 69 (1994), 483-499.}
\bibitem{florit-ziller}{{\sc L. Florit \and W. Ziller}, \emph{Manifolds with conullity at most two as graph manifolds}, to appear.}
\bibitem{gromov}{{\sc M. Gromov}, \emph{Manifolds of negative curvature}, J. Differential Geom. 13 (1978), no. 2, 223-230.}
\bibitem{ha}{{\sc U. Hamenst\"adt}, \emph{A geometric characterization of negatively curved locally symmetric spaces.} J. Differential Geom. 34 (1991), no. 1, 193-221.}
\bibitem{kapo}{{\sc M. Kapovich}, \emph{Flats in $3$-manifolds}, Ann. Fac. Sci. Toulouse Math. (6) 14 (2005), no. 3, 459-499.}
\bibitem{nomizu}{{\sc K. Nomizu}, \emph{On local and global existence of Killing vector fields}, Ann. of Math. (2) 72 1960 105--120.}
\bibitem{scshsp}{{\sc B. Schmidt, R. Shankar \and R. Spatzier}, \emph{Positively curved manifolds with large spherical rank},  Comment.\ Math.\ Helv., to appear.}
\bibitem{scwo}{{\sc B. Schmidt \and J. Wolfson}, \emph{Three-manifolds with constant vector curvature}, Indiana Univ. Math. J. 63 (2014), no. 6, 1757--1783.}
\bibitem{se}{{\sc K. Sekigawa}, \emph{On the Riemannian manifolds of the form $B \times_f F$}, Kodai Math. Sem. Rep., 26 (1975), 343-347.}
\bibitem{shspwi}{{\sc K. Shankar, R. Spatzier \and B. Wilking}, \emph{Spherical rank rigidity and Blaschke manifolds}, Duke Math. Journal, 128 (2005), 65-81.}
\bibitem{spst}{{\sc R. Spatzier \and M. Strake}, \emph{Some examples of higher rank manifolds of nonnegative curvature},  Comment.\ Math.\ Helv.\ 65 (1990), 299-317.}
\bibitem{sus}{{\sc H. J. Sussmann}, \emph{Orbits of families of vector fields and integrability of distributions}. Trans.\ Amer.\ Math.\ Soc.\ 180 (1973), 171-188.}
\bibitem{wald}{{\sc F. Waldhausen}, \emph{Eine Klasse von 3-dimensionalen Mannigfaltigkeiten}. I, II.
Invent. Math. 3 (1967), 308-333; ibid. 4 1967 87-117.}
\bibitem{wa}{{\sc J. Watkins}, \emph{The higher rank rigidity theorem for manifolds with no focal points}, Geom. Dedicata 164 (2013), no.1, 319-349.}
\end{thebibliography}
\end{document}